\documentclass[11pt,a4paper]{amsart}
\usepackage{amssymb,amsmath,layout}

\theoremstyle{plain}
\newtheorem{thrm}{Theorem}[section]
\newtheorem{lemma}[thrm]{Lemma}
\newtheorem{prop}[thrm]{Proposition}

\newtheorem{rmrk}[thrm]{Remark}
\newtheorem{dfn}[thrm]{Definition}

\begin{document}
% begin top matter
% ***************** macroes needed for this paper ************************

\newcommand{\SL}{\mathcal L^{1,p}( D)}
\newcommand{\Lp}{L^p( Dega)}
\newcommand{\CO}{C^\infty_0( \Omega)}
\newcommand{\Rn}{\mathbb R^n}
\newcommand{\Rm}{\mathbb R^m}
\newcommand{\R}{\mathbb R}
\newcommand{\Om}{\Omega}
\newcommand{\Hn}{\mathbb H^n}
\newcommand{\aB}{\alpha B}
\newcommand{\eps}{\ve}
\newcommand{\BVX}{BV_X(\Omega)}
\newcommand{\p}{\partial}
\newcommand{\IO}{\int_\Omega}
\newcommand{\bG}{\boldsymbol{G}}
\newcommand{\bg}{\mathfrak g}
\newcommand{\bz}{\mathfrak z}
\newcommand{\bv}{\mathfrak v}
\newcommand{\Bux}{\mbox{Box}}
\newcommand{\e}{\ve}
\newcommand{\X}{\mathcal X}
\newcommand{\Y}{\mathcal Y}
\newcommand{\W}{\mathcal W}

\numberwithin{equation}{section}

\newcommand{\RN} {\mathbb{R}^N}
\newcommand{\Sob}{S^{1,p}(\Omega)}
\newcommand{\Dxk}{\frac{\partial}{\partial x_k}}
\newcommand{\Co}{C^\infty_0(\Omega)}
\newcommand{\Je}{J_\ve}
\newcommand{\beq}{\begin{equation}}
\newcommand{\bea}[1]{\begin{array}{#1} }
\newcommand{\eeq}{ \end{equation}}
\newcommand{\ea}{ \end{array}}
\newcommand{\eh}{\ve h}
\newcommand{\Dxi}{\frac{\partial}{\partial x_{i}}}
\newcommand{\Dyi}{\frac{\partial}{\partial y_{i}}}
\newcommand{\Dt}{\frac{\partial}{\partial t}}
\newcommand{\aBa}{(\alpha+1)B}
\newcommand{\GF}{\psi^{1+\frac{1}{2\alpha}}}
\newcommand{\GS}{\psi^{\frac12}}
\newcommand{\HFF}{\frac{\psi}{\rho}}
\newcommand{\HSS}{\frac{\psi}{\rho}}
\newcommand{\HFS}{\rho\psi^{\frac12-\frac{1}{2\alpha}}}
\newcommand{\HSF}{\frac{\psi^{\frac32+\frac{1}{2\alpha}}}{\rho}}
\newcommand{\AF}{\rho}
\newcommand{\AR}{\rho{\psi}^{\frac{1}{2}+\frac{1}{2\alpha}}}
\newcommand{\PF}{\alpha\frac{\psi}{|x|}}
\newcommand{\PS}{\alpha\frac{\psi}{\rho}}
\newcommand{\ds}{\displaystyle}
\newcommand{\Zt}{{\mathcal Z}^{t}}
\newcommand{\XPSI}{2\alpha\psi \begin{pmatrix} \frac{x}{|x|^2}\\ 0 \end{pmatrix} - 2\alpha\frac{{\psi}^2}{\rho^2}\begin{pmatrix} x \\ (\alpha +1)|x|^{-\alpha}y \end{pmatrix}}
\newcommand{\Z}{ \begin{pmatrix} x \\ (\alpha + 1)|x|^{-\alpha}y \end{pmatrix} }
\newcommand{\ZZ}{ \begin{pmatrix} xx^{t} & (\alpha + 1)|x|^{-\alpha}x y^{t}\\
     (\alpha + 1)|x|^{-\alpha}x^{t} y &   (\alpha + 1)^2  |x|^{-2\alpha}yy^{t}\end{pmatrix}}
\newcommand{\norm}[1]{\lVert#1 \rVert}
\newcommand{\ve}{\varepsilon}

\title[Quantitative uniqueness for elliptic  equations, etc.]{Quantitative uniqueness for  elliptic equations at the boundary of $C^{1,\operatorname{Dini}}$ domains}

\author{Agnid Banerjee}
\address{Department of Mathematics\\University of California, Irvine\\
CA 92697} \email[Agnid Banerjee]{agnidban@gmail.com}

\author{Nicola Garofalo}
\address{Dipartimento di Ingegneria Civile, Edile e Ambientale (DICEA) \\ Universit\`a di Padova\\ 35131 Padova, ITALY}
\email[Nicola Garofalo]{rembdrandt54@gmail.com}

\thanks{Second author supported in part by a grant ``Progetti d'Ateneo, 2014'' of the University of Padova.}

%
% 
% AMS information
%
\keywords{}
\subjclass{}

\maketitle
% end top matter

\begin{abstract}
Based on a variant of the frequency function approach of  Almgren, we  establish  an optimal upper bound on the  vanishing order of solutions to variable coefficient Schr\"odinger equations at a portion of the  boundary of a  $C^{1,\operatorname{Dini}}$ domain. Such bound provides a quantitative form of strong unique continuation at the boundary. It can be thought of as a boundary analogue of an interior result  recently obtained by Bakri and Zhu for the standard Laplacian. 

\end{abstract}

\section{Introduction}
We say that the vanishing order of a function $u$ is $\ell$ at $x_0$, if $\ell$ is the largest  integer such that $D^{\alpha} u=0$ for all $|\alpha| \leq \ell$, where $\alpha$ is a multi-index. In the papers \cite{DF}, \cite{DF1}, Donnelly and Fefferman showed that if $u$ is an eigenfunction with eigenvalue  $\lambda$ on a  smooth, compact and connected $n$-dimensional Riemannian manifold $M$, then  the maximal vanishing order of $u$ is less than $C \sqrt{\lambda}$ where $C$ only depends on the manifold $M$. This order of vanishing is sharp. If, in fact, we consider $M = \mathbb S^n \subset \R^{n+1}$, and we take the spherical harmonic $Y_\kappa$ given by the restriction to $\mathbb S^n$ of the function $f(x_1,...,x_n,x_{n+1}) = \Re (x_1 + i x_2)^\kappa$, then one has $\Delta_{\mathbb S^n} Y_\kappa = - \lambda_\kappa Y_\kappa$, with $\lambda_\kappa = \kappa(\kappa+n-2)$, and the order of vanishing of $Y_\kappa$ at the North pole $(0,...,0,1)$ is precisely $\kappa = C \sqrt {\lambda_\kappa}$. 

In his work  \cite{Ku}  Kukavica considered the more general problem
\begin{equation}\label{e1}
\Delta u = V(x) u,
\end{equation}
where $V\in W^{1, \infty}$, and showed that the maximal vanishing order  of $u$ is bounded above by $C( 1+ ||V||_{W^{1, \infty}})$. He also conjectured that the rate of vanishing order of $u$ is less than or equal to $C(1+ ||V||_{L^{\infty}}^{1/2})$, which agrees with the Donnelly-Fefferman result when $V = - \lambda$. Employing Carleman estimates,  Kenig in \cite{K} showed that the rate of  vanishing order  of $u$  is less than $C(1+ ||V||_{L^{\infty}}^{2/3})$, and that furthermore the exponent $\frac{2}{3}$ is sharp for complex  potentials $V$. 

Recently, the rate of vanishing order of $u$ has been shown to be less than $C(1+ ||V||_{W^{1, \infty}}^{1/2})$  independently by Bakri in \cite{Bk} and Zhu in \cite{Zhu1}. Bakri's approach is based on  an extension of the Carleman method in \cite{DF}. In this connection, we also quote the recent interesting paper by R\"uland \cite{Ru}, where Carleman estimates are used to obtain related quantitative unique continuation results for nonlocal Schr\"odinger operators such as $(-\Delta)^{s/2} + V$. On the other hand, Zhu's approach is based on a variant of the frequency function approach  employed by Garofalo and Lin in \cite{GL1}, \cite{GL2}), in the context of strong  unique  continuation problems. Such variant consists in studying the growth properties of the following average of the Almgren's height function 
\[
H(r) = \int_{B_r(x_0)} u^2 (r^2 - |x-x_0|^2)^\alpha dx, \ \ \ \ \ \ \ \ \alpha> - 1,
\]
first introduced by Kukavica in \cite{Ku2} to study quantitative unique continuation and vortex degree estimates for solutions of the Ginzburg-Landau equation. 

In \cite{Bk} and  \cite{Zhu1} it was assumed that $u$ be a solution in $B_{10}$ to 
\begin{equation}
\Delta u= Vu,
\end{equation}
with $||V||_{W^{1, \infty}} \leq M$ and  $||u||_{L^{\infty}} \leq C_0$, and that furthermore $\sup_{B_1} |u| \geq 1$. Then, it was proved that $u$ satisfies the sharp growth estimate
\begin{equation}\label{gn}
||u||_{L^{\infty}(B_r)} \geq B r^{C(1 + \sqrt{M})},
\end{equation}
where $B, C$ depend only on $n$ and $C_0$. 

In this note  we establish an analogous  quantitative uniqueness result at the boundary of a $C^{1,\operatorname{Dini}}$ domain for elliptic operators with variable coefficients in a borderline situation.  Our main result can be stated as follows.

\begin{thrm}\label{main}
Let $\Om$ be a $C^{1,\operatorname{Dini}}$ domain and suppose that $A(x) \in C^{0,1}(\overline{\Om})$ be such that
there exists $\lambda > 0$ for which for every $x\in \overline \Om$ and every $\xi\in \Rn$ one has
\begin{equation}\label{hypA2}
\lambda |\xi|^2 \leq <A(x)\xi,\xi>  \leq  \lambda^{-1} |\xi|^2,
\end{equation}
and for which
\begin{equation}\label{hypA1}
||A||_{C^{0,1}(\overline{\Om})} \leq K.
\end{equation} 
Let $V\in W^{1,\infty}(\Om)$, with $||V||_{W^{1, \infty}(\Om)} \le M$. 
Let $u\in W^{1,2}(\Om)\cap L^\infty(\Om)$ be a weak solution in $\Om$ to the equation
\begin{equation}\label{m1}
\operatorname{div}(A(x) Du)  = V(x) u,
\end{equation}
and assume that $||u||_{L^{\infty}(\Om)} \leq C_0$.
Given  an open set $\Gamma \subset \partial \Om$, suppose that $u \in C(\Om \cup \Gamma)$ and that  $u$ vanishes on $\Gamma$.  Let $x_0 \in \Gamma$ be such that $\partial \Om \cap B_2(x_0)  \subset \Gamma$, and for which $\sup_{\Om \cap B_1(x_0)} |u| \geq 1$. Then, there exist constants  $B, C$, and $R_0>0$, depending on $n,  \lambda, K, C_0,$ and the $C^{1,\operatorname{Dini}}$-character of $\Om$,  such that for all $0<r<R_0$ one has
\begin{equation}\label{m2}
||u||_{L^{\infty}(\Om \cap B_r(x_0))} \geq B r^{C(1 + \sqrt{M})}.
\end{equation}

\end{thrm}

In the literature $C^{1,\operatorname{Dini}}$ domains are often referred to as Dini domains, see \cite{KN}, \cite{AE}. We mention that the strong unique continuation property at the boundary for $C^{1,\operatorname{Dini}}$ domains and for elliptic equations, but not the order of vanishing,  was established in \cite{AE} by adapting the frequency function approach in \cite{GL1}, \cite{GL2}. We also refer to \cite{KN} for a simpler proof of the result in \cite{AE}. In \cite{AE} it was shown that $e^{Cr} N(r)$ is monotone, where $N(r)$ is the frequency used by Garofalo and Lin.  However, similarly to the interior estimates in \cite{GL2}, the constant $C$ depends on the norm of $V$,  and since it appears in the exponent, it would  only give an upper bound on the vanishing order of $u$  proportional to $e^{CM}$, which is not the optimal bound  $1 + \sqrt{M}$. As we have mentioned above, Kukavica in \cite{Ku}  was able to remove the dependence on $V$ from the exponential by considering a variant of the frequency, but he was only able to obtain an upper bound proportional to  $C(1+ ||V||_{W^{1, \infty}})$.  

In order to establish  the estimate \eqref{m2} in Theorem \ref{main} above we borrow some of the ideas in \cite{Ku2} and \cite{Zhu1}, and adapt them to our different situation. Simultaneously, we also use some ideas from \cite{KN}. Nevertheless, the proof of Theorem \ref{main} entails a substantial amount of novel work. This is mainly due to the fact that we are working at the boundary, and that, unlike \cite{Ku2}, \cite{KN} and \cite{Zhu1}, we are dealing with a variable coefficient operator. This forces one to deal with some delicate uniformity matters which arise at several steps in the process. Loosely speaking, the essential idea of the proof of Theorem \ref{main} is to obtain some analogue of the interior estimates in \cite{Zhu1} first for starshaped domains. Then, for a given $C^{1,\operatorname{Dini}}$ domain, at each scale $r$, one can find a starshaped domain as in \cite{KN}  where such an estimate can be obtained, and then we iterate  the estimate  at every scale  by crucially  making use of the Dini modulus of continuity of the normal at the boundary. This allows us  to  obtain uniform bounds on the constants involved, see Section \ref{S:gen} below. We mention as well that, as opposed to the case when $\Delta u=0$, that was dealt with in \cite{KN},  in our situation the presence of the potential $V$ does not allow a pure monotonicity of the modified frequency $N(r)$ defined in \eqref{genf} below,  but that of a perturbed one as in Theorem \ref{vimp}. Therefore,  the scalings  have to be chosen differently  from \cite{KN} in the corresponding iteration argument of our  proof.

The paper is organized as follows. In Section \ref{S:moving}  we introduce a change of coordinates  which allows  to normalize the situation  needed to prove  the relevant monotonicity of the variant of Almgren's frequency  as in  Theorem \ref{vimp}. In Section \ref{S:dini} we introduce the relevant framework and also collect some local geometric properties of $C^{1,\operatorname{Dini}}$ domains which play a crucial  role in the rest of the paper. In this regard, we would like to mention that the smallness of the deviation of the normal to the boundary is the most important property that is used in our proof of Theorem \ref{main}. Such property allows to obtain the uniformity in our most technical result, Lemma \ref{imp2} below.  In Section \ref{S:fvf} we establish our main result about the monotonicity of the frequency, see Theorem \ref{vimp} below. As a consequence of such result, in Section \ref{S:3s} we derive some three-sphere lemmas at the boundary of star-shaped domains. In Section \ref{S:gm} we establish two basic growth lemmas that constitute the backbone of the proof of Theorem \ref{main}. Finally in Section \ref{S:gen} we prove our main result, Theorem \ref{main}. 

\medskip

\noindent \textbf{Acknowledgment:} We thank Gary Lieberman for kindly providing us with the reference \cite{Li}.

\section{A basic normalization and uniformity matters}\label{S:moving}

In this section we introduce a change of coordinates that will play a crucial role in the proof of Theorem \ref{main}. To elucidate this aspect we mention that the proof of Theorem \ref{main} differs substantially depending on whether the matrix $A(x) \equiv I_n$ (hereafter in this paper $I_n$ indicates the identity matrix in $\Rn$), in which case we have the standard Laplacian, or we are dealing with a genuinely variable coefficient operator. The crux of the question is that, in the latter case, the proof of the basic first variation formulas, Lemma \ref{L:fvH} and Proposition \ref{P:fvegen}, uses the normalizing assumption $A(0) = I_n$ in a crucial way. As a consequence,  this hypothesis also permeates the important monotonicity Theorem \ref{vimp}, and the ensuing three-sphere Lemmas \ref{L:3sH} and \ref{L:3sh}. However, to proceed in the analysis we need to apply these results at appropriate interior points in $\overline \Om$, where the normalizing hypothesis is not necessarily valid. The main purpose of this section is to show that, by a suitable change of coordinates, we can accomplish this situation without changing in any quantitative way the hypothesis of Theorem \ref{main}. By this we mean that we can go back and forth with our change of coordinates, while at the same time keeping under control some important uniformity aspects of the estimates involved.

Suppose that $\Om$ is a bounded open set and that $A(x)$ is a matrix-valued function in $\overline \Om$ satisfying the assumption of Theorem \ref{main} above. For a given point $z_0\in \overline \Om$ suppose that we are in the situation that $A(z_0)$ is not the identity matrix $I_n$. We consider the affine transformation $T_{z_0}: \Rn \to \Rn$ defined by
\begin{equation}\label{T}
T_{z_0}(x) = A(z_0)^{-1/2}(x-z_0).
\end{equation}
$T_{z_0}$ is a bijection from $\Om$ onto its image $\Om_{z_0} = T_{z_0}(\Om)$.
For a given function $f:\Om \to \R$ we consider the function $f_{z_0} : \Om_{z_0} \to \R$ defined by 
\begin{equation}\label{tran}
f_{z_0}(y) = f \circ T_{z_0}^{-1}(y) = f(z_0+A^{1/2}(z_0)y),\ \ \ y\in \Om_{z_0},
\end{equation}
and the matrix-valued function defined in $\Om_{z_0}$ as follows
\begin{equation}\label{cv}
A_{z_0}(y) =A^{-1/2}(z_0)A(z_0+A^{1/2}(z_0)y)A^{-1/2}(z_0).
\end{equation}
In a standard way one verifies that if $u$ is a weak solution to \eqref{m1} in $\Om$, then $u_{z_0}$ is a weak solution in $\Om_{z_0}$ to  
\begin{equation}\label{neweq}
L_{z_0} u_{z_0}= \operatorname{div}(A_{z_0}(y) D u_{z_0}) = V_{z_0}(y) u_{z_0}.
\end{equation}
It is important to notice that the potential $V_{z_0}$ satisfies in $\Om_{z_0}$ the same differentiability assumption as $V$, and that moreover
\begin{equation}\label{VV}
||V_{z_0}||_{W^{1,\infty}(\Om_{z_0})} \le C M,
\end{equation}
where $M$ is the bound on $||V||_{W^{1,\infty}(\Om)}$ in Theorem \ref{main}, and $C = C(\lambda)>0$. 

Hereafter in this paper, when we say that a constant is universal, we mean that it depends exclusively on $n$, on the ellipticity bound $\lambda$ on $A(x)$, see \eqref{hypA2} above, on the Lipschitz bound $K$ in \eqref{hypA1}, and on the $C^{1,\operatorname{Dini}}$-character of the domain $\Om$. However, a universal constant will never depend on the bound $M$ on the $W^{1,\infty}$ norm of $V$ in Theorem \ref{main}. Likewise, we will say that $O(1)$, $O(r)$, etc. are universal if $|O(1)| \le C$, $|O(r)|\le C r$, etc., with $C\ge 0$ universal.

Notice that since $0 = T_{z_0}(z_0) \in \Om_{z_0}$, the change of variable $y = T_{z_0}(x)$ allows us to assume that $z_0=0$ and, more importantly, that \eqref{dev1} below hold since, by construction, we have $A_{z_0}(0)=I_n$. 
%Furthermore, if  we define     
%\begin{equation}\label{nmu}
%\mu_{z_0}(y)= \frac{<A_{z_0}(y) y,y>}{|y|^2},\ \ \ \ \ \ \ y\in \Om_{z_0},
%\end{equation}
%then we have $\mu_{z_0}(0) = 1$. 

Before proceeding we note explicitly that in passing from the matrix $A$ in $\Om$ to the matrix $A_{z_0}$ in $\Om_{z_0}$ the uniform bounds on the ellipticity change from $\lambda$ to $\lambda^2$. We have in fact for every $y\in \Om_{z_0}$ and any $v\in \Rn$
\begin{equation}\label{ne}
\lambda^2 |v|^2\ \le\ <A_{z_0}(y) v,v>\ \le\ \lambda^{-2} |v|^2.
\end{equation}
Moreover, the hypothesis \eqref{hypA2} above implies that for every $z_0\in \overline \Om$ and $x, p \in \Rn$
\begin{align}\label{sqrtA}
& \lambda^{1/2} |x - p| \le |A^{1/2}(z_0) (x-p)| \le \lambda^{-1/2} |x-p|,
\\
& \lambda^{1/2} |x-p| \le |A^{-1/2}(z_0) (x-p)| \le \lambda^{-1/2} |x-p|.
\notag
\end{align}
We can rewrite the second inequality in \eqref{sqrtA} in the following way
 \begin{align}\label{bA}
& \lambda^{1/2} |x - p| \le |T_{z_0}(x)- T_{z_0}(p)| \le \lambda^{-1/2} |x-p|,
\end{align}
or, equivalently,
\begin{equation}\label{cont}
B_{\sqrt \lambda r}(T_{z_0}(p)) \subset T_{z_0}(B_r(p)) \subset  B_{\frac{r}{\sqrt \lambda}}(T_{z_0}(p)),
\end{equation}
for any $p\in \Rn$ and $r>0$. The inclusion \eqref{cont} will play a pervasive role in the proof of the central Lemma \ref{imp2} in Section \ref{S:gm} below.

 Finally, we note that the matrix-valued function $y\to A_{z_0}(y)$ is Lipschitz continuous in $\Om_{z_0}$, and in fact from \eqref{sqrtA} and \eqref{hypA1} above we have
\begin{equation}\label{newlip}
||A_{z_0}(y) - A_{z_0}(y')|| \le K' |y - y'|,\ \ \ \ \ \ \ \ \ \ y, y'\in \Om_{z_0},
\end{equation}
where $K ' = \lambda^{-3/2} K$, where $K$ is as in \eqref{hypA1}.

\section{Local geometry of $C^{1,\operatorname{Dini}}$ domains}\label{S:dini}

 In this section we collect some local properties of $C^{1,\operatorname{Dini}}$ domains that play a pervasive role in this paper.
 Throughout the discussion, we denote by $B_{r}(x)$ the Euclidean ball of radius $r$ with center at $x$. A ball centered at $0$ will be simply denoted by $B_r$.  

\begin{dfn}\label{2.1}
A connected bounded open set $\Om \subset \Rn$ is called a \emph{$C^{1,\operatorname{Dini}}$  domain} if for each point $x_0 \in \partial \Om$ there exist a local coordinate system $(x', x_n) \in \R^{n-1} \times \R$, $R_0>0$, and a function $\varphi: \R^{n-1} \to R$ such that:
\begin{itemize}
\item[i)] $B_{R_0}(x_0) \cap \Om= \{ x \in B_{R_0}(x_0) : x_n < \varphi(x') \}$;
\item[ii)] $B_{R_0}(x_0) \cap \partial \Om= \{ x \in B_{R_0}(x_0) : x_n = \varphi(x') \}$;
\item[iii)] $|D'\varphi(x'_1) - D'\varphi(x'_2)| \leq \psi(|x'_1- x'_2|)$ where $\psi$ satisfies 
\begin{equation}\label{dini}
\int_{0}^2 \frac{\psi(r)}{r} dr < \infty.
\end{equation}
\end{itemize}
\end{dfn}
We notice that the outer unit normal at a point $(x',x_n)\in B_{R_0}(x_0) \cap \partial \Om$ is given by
\begin{equation}\label{normal}
\nu = \left(- \frac{D'\varphi}{\sqrt{1+|D'\varphi|^2}}, \frac{1}{\sqrt{1+|D'\varphi|^2}}\right).  
\end{equation}

We consider a given $C^{1,\operatorname{Dini}}$ domain $\Om$ and, for a fixed point $x_0\in \partial \Om$, we let $R_0$, $\varphi$ and $\psi$ be as in Definition \ref{2.1} above. Without loss of generality, by a translation we can suppose that $x_0 = 0$, and this forces $\varphi(0) = 0$. We denote by $\nu(x)$ the outward unit normal at  $x\in \partial \Om$. Without loss of generality, by a rotation we can also assume that $D'\varphi(0) = 0$, so that $\nu(0)= e_n = (0,...,0,1)$. Thereby, by possibly restricting its value, we can assume that $R_0>0$ is such that
\[
\underset{|x'|\le R_0}{\sup}\  \sqrt{1+|D'\varphi(x')|^2} \le \frac 32.
\] 

Throughout this paper we assume  that there exists a  non-decreasing function $r \to \Lambda(r)$ such that for every $r \in (0,R_0)$ we have
\begin{equation}\label{e'''}
\underset{x_1, x_2 \in  \partial \Om\cap B_r}{\sup}\ |\nu(x_2) - \nu(x_1)| \leq \Lambda(r).
\end{equation}
Also  from \eqref{normal} above and the Dini assumption \eqref{dini} on $D' \varphi$, we can take $\Lambda(r)$ such that 
\begin{equation}\label{e20}
\int_{0}^{R_0} \frac{\Lambda(r)}{r} dr < \infty.
\end{equation}
It should also be clear that, by the way we pick the function $\Lambda$  in \eqref{e'''}, there is no loss of generality in assuming that
\begin{equation}\label{sqrt}
\Lambda(r) \ge \sqrt r,\ \ \ \ \ \ \ \ \ \  \ 0<r\le R_0.
\end{equation} 
We now employ considerations similar to those in the proof of Lemma 3.2 in \cite{KN}. Fix any $r \in (0, R_0)$ sufficiently small.  Let $x_1 = (x'_1, \varphi(x'_1))$ and $x_2 = (x'_2, \varphi(x'_2))$ be two arbitrary points in $B_r \cap \partial \Om$. Then, one has 
\[
|D' \varphi(x'_j)| \le |\nu(x_j) - \nu(0)| \sqrt{1+|D'\phi(x_j')|^2} \le \underset{|x'|\le r}{\sup}\  \sqrt{1+|D'\varphi(x')|^2}\  \Lambda(r) \le \frac 32 \Lambda(r),
\]
provided that $0<r\le R_0$. By the mean value theorem we thus have for $0<r\le R_0$
\[
|\varphi(x_1') - \varphi(x_2')|\le \frac{3}{2} \Lambda(r) |x_1'-x_2'|.
\]
Since $\Lambda(r) \to 0 $ as $ r \to 0$, we can  assume $R_0$ is small enough so that 
\begin{equation}\label{1000}
R_0<1,\ \ \ \text{and}\ \ \ \Lambda(R_0) < \frac1{1000}.
\end{equation}
This is the first place where we use the fact that $\Om$ be $C^{1,\operatorname{Dini}}$. 
We state a simple lemma that will be needed in Section \ref{S:gm} below.

\begin{lemma}\label{L:star}
Let  $a= 4 \Lambda(r)r$, and consider the following interior point $y_0= - a \nu(0) = (0,-a)$ associated with $x_0 = 0$. Then, for $0<r\le R_0$ the set $\Om \cap B_{r-a}(y_0)$ is star-shaped with respect to $y_0$.
In addition, the following quantitative form of star-shapedness holds: for $0 <r \le R_0$ and $x \in \partial \Om \cap B_{r}$, we have
\begin{equation}\label{qtsr}
\frac{r \Lambda(r)}{2}\  \le\ <x-y_0, \nu(x)>\ \le\ 10 r \Lambda(r).
\end{equation} 
\end{lemma}

\begin{proof}
Although the star-shapedness of $\Om \cap B_{r-a}(y_0)$ has already been observed in the proof of Lemma 3.2 in \cite{KN}, it will also follow from \eqref{qtsr}, to which proof now we turn. Let $x= (x', \varphi(x'))$. One has 
\[
|\varphi(x')| \leq \frac{3 \Lambda(r) r}{2}, \ \ \ \ \ \ \ \ |\nu(x) - \nu(0)| \leq \Lambda(r).
\]
Then,
\begin{align*}
& <x-y_0, \nu(x)> = <x-y_0,\nu(0)> + <x-y_0,\nu(x) - \nu(0)>
\\
& = \varphi(x') + a + <(x', \varphi(x') + a),\nu(x) - \nu(0)>  
\notag\\
& \geq 4\Lambda(r) r  -  \frac{3 \Lambda(r) r}{2}  -  r\Lambda(r) - 4 (\Lambda(r))^2 r - \frac{3 (\Lambda(r))^2 r}{2} \geq \frac{r \Lambda(r)}{2}, 
\notag
\end{align*}
where in the last inequality we have used \eqref{1000}. This establishes the bound from below in \eqref{qtsr}. In a similar way, one can show the bound from above.

\end{proof}

We next establish a stronger version of Lemma \ref{L:star} which is needed in the applications of Theorem \ref{vimp} in Section \ref{S:gm}. Given $x_0\in \partial \Om$, we continue to assume as in Lemma \ref{L:star} that $x_0 = 0$, $\nu(0) = e_n$, and as before we consider the interior point $y_0 = - a \nu(0)$ associated with $x_0 = 0$. We now use the transformation \eqref{tran}  above, with $z_0$ replaced by $y_0$, to send $y_0$ to $0$. In doing so, the point $x_0=0$ will clearly go to the point $p'=T_{y_0}(0)$. It is immediate to verify from \eqref{bA} that 
\begin{equation}\label{pprime}
|p'| = |T_{y_0}(0) - T_{y_0}(y_0)| \le \lambda^{-1/2} |y_0| = \lambda^{-1/2} a.
\end{equation}
The transformed  domain $\Om_{y_0} = T_{y_0}(\Om)$ will have $p'$ on its boundary (this point is the image of the point $x_0 = 0 \in \partial \Om$), whereas $0 = T_{y_0}(y_0)$ will be in the interior of $\Om_{y_0} $. In $\Om_{y_0}$ we thus have for the transformed matrix $A_{y_0}$, that $A_{y_0}(0) = I_n$. By slightly abusing the notation, we continue denoting by  $\Lambda$, instead of $\Lambda_{y_0}$, the function in \eqref{e'''} for the domain $\Om_{y_0}$. We notice that $\Lambda_{y_0}$ differs only by a multiplicative constant from the original $\Lambda$.

\begin{lemma}\label{L:gens}
For every $0<r \le R_0$ the set $\Om_{y_0} \cap B_{\sqrt \lambda (r-a)}$ is \emph{generalized star-shaped} with respect to $0$ and the matrix $A_{y_0}$, in the sense that for every $y\in \partial \Om_{y_0} \cap B_{\sqrt \lambda (r-a)}$ one has
\begin{equation}\label{St1}
<A_{y_0}(y) y,\tilde N(y)>\ \ge\ 0,
\end{equation}
where $\tilde N(y)$ denotes an outer normal in $y$.
\end{lemma}

\begin{proof}

We observe that by Lemma \ref{L:star} we know that for every $0<r\le R_0$, the set $\Om \cap B_{r-a}(y_0)$ is star-shaped with respect to $y_0$ and \eqref{qtsr} above hold.
We first claim that:
\begin{equation}\label{tss}
\Om_{y_0} \cap B_{\sqrt \lambda (r-a)}  \ \  \text{is star-shaped with respect to}\  0 \ \ \text{for}\ 0<r\le R_0. 
\end{equation}
In order to prove the claim, we notice that by the left-inclusion in \eqref{cont} above, it suffices to verify that
\begin{equation}\label{tss2}
\Om_{y_0} \cap T_{y_0}(B_{r-a}(y_0)) \ \  \text{is star-shaped with respect to}\  0 \ \ \text{for}\ 0<r\le R_0. 
\end{equation}
Now, to prove \eqref{tss2} it suffices to prove that if $0<r\le R_0$, and if $\tilde N(y)$ denotes a (non-unit) outer normal to $\partial \Om_{y_0} \cap T_{y_0}(B_{r-a}(y_0))$, then we have $<y,\tilde N(y)> \ge 0$. It is easy to recognize that $\tilde N(y) = A(y_0)^{1/2} N(T_{y_0}^{-1}(y))$, where $N$ indicates a non-unit outer normal field on $\partial \Om$. We thus have
\begin{align*}
& <y,\tilde N(y)>\ =\ <T_{y_0}(T_{y_0}^{-1}(y)),A(y_0)^{1/2} N(T_{y_0}^{-1}(y))>\ 
 =\ <A(y_0)^{1/2}T_{y_0}(T_{y_0}^{-1}(y)),N(T_{y_0}^{-1}(y))> 
 \\
 &\ =\ <T_{y_0}^{-1}(y)-y_0,N(T_{y_0}^{-1}(y))> \ \ge\ 0,
\end{align*}
where in the last inequality we have used  \eqref{qtsr} which  holds for $x \in \partial \Om \cap B_r$ and therefore by triangle inequality also holds for $x = T_{y_0}^{-1}(y) \in \partial \Om \cap B_{r-a} (y_0)$.
This proves \eqref{tss2}, and therefore \eqref{tss}. 

We next want to show that, by possibly further restricting the value of $R_0\in (0,1)$, we can accomplish \eqref{St1}. For this we are going to use the quantitative star-shapedness expressed by\eqref{qtsr} above. Let $y\in \partial \Om_{y_0} \cap B_{\sqrt \lambda (r-a)}(0)$. Since $A_{y_0}(0)=I_n$, we have $A_{y_0}(y) = I_n + (A_{y_0}(y) - A_{y_0}(0))$. Now, \eqref{newlip} and $|p'| \leq \lambda^{-1/2} a$ (see \eqref{pprime} above) give 
\[
||A_{y_0}(y) -A_{y_0}(0)|| \le K' |y| \le K' |y-p'| + K' |p'| \le K' (\sqrt \lambda (1- 4\Lambda(r)) + \frac{4}{\sqrt \lambda} \Lambda(r)) r \le C r, 
\]
for some universal $C>0$. On the other hand, \eqref{qtsr}  gives 
\[
10 r\Lambda(r)\ \geq\ <y,\tilde N(y)>\  \geq\ \frac{r \Lambda(r)}{2}.
\]
Therefore, by possibly restricting further $R_0$, we have for $0<r\le R_0$
\begin{align*}
& <A_{y_0}(y) y,\tilde N(y)>= < y,\tilde N(y)> + <(A_{y_0}(y) - A_{y_0}(0))y,\tilde N(y)>
\\
& \geq\ \frac{r \Lambda(r)}{2} - C r^2\ \geq\ \frac{r^{3/2}}{2} - C r^2\ \ge\ 0,
\end{align*}
where in the second to the last inequality we have used \eqref{sqrt} above. This proves  \eqref{St1}.

\end{proof}

\section{First variation formulas and adjusted monotonicity of the frequency}\label{S:fvf}

The principal objective of this section is establishing the monotonicity Theorem \ref{vimp} below. We begin with some preliminary material. 
Given a point $z_0 \in \overline{\Om}$, for $x\in \Om$ we let $r_{z_0}(x) = |x - z_0|$.  Also, we will adopt the summation convention over repeated indices. For  $z_0\in \overline \Om$, we let $ B_{z_0}(x) = A(x) - A(z_0)$. Let us notice that \eqref{hypA1} above gives 
\begin{equation}\label{B}
||B_{z_0}(x)|| \le C |x - z_0|,
\end{equation}
with $C>0$ universal.
When $z_0$ is fixed in a certain context, we will routinely write $B(x)$, instead of $B_{z_0}(x)$.
The next lemma expresses a simple, yet important fact.

\begin{lemma}\label{L:div(Agradr)} 
Suppose $A(z_0) = I_n$. Then, for $x\not= z_0$, one has 
\begin{equation}\label{perturbation}
L r_{z_0}= \emph{div}(A(x)D r_{z_0})=\frac{n-1}{r_{z_0}} +O(1),
\end{equation}
where $O(1)$ is universal.
In particular, $Lr_{z_0} \in L^1_{loc}(\Rn)$.
\end{lemma}

\begin{proof} 
We have 
\[
\text{div}(A(x)D r_{z_0})=\text{div}(A(z_0)D r_{z_0})+\text{div}(D r_{z_0})=\Delta r_{z_0}+\text{div}(B(x)D r_{z_0})=\frac{n-1}{r_{z_0}}+\text{div}(B(x) D r_{z_0}).
\]
Now, if  $B(x) = [b_{ij}(x)]$, we have
\[
\text{div}(B(x) D r_{z_0})= D_i\left(b_{ij}(x) D_j r_{z_0}\right)= D_i(b_{ij}) D_j r_{z_0} + b_{ij} D_{ij} r_{z_0}.
\]
From \eqref{hypA1} and the Rademacher-Stepanov theorem we have $D_i(b_{ij})= O(1)$, with $O(1)$ universal. By \eqref{B} we find $b_{ij} D_{ij} r_{z_0} = O(1)$, with $O(1)$ universal. In conclusion, div$(B(x)D r_{z_0}) = O(1)$. This gives
\[
\operatorname{div}(A(x)D r_{z_0})=\frac{n-1}{r_{z_0}}+O(1).
\]

\end{proof}

We next introduce the conformal factor
\begin{equation}\label{mu}
\mu_{z_0} (x)= <A(x) D r_{z_0}(x), D r_{z_0}(x)> = \frac{<A(x)(x-z_0), x-z_0>}{|x-z_0|^2}.
\end{equation}
Let us observe explicitly that when $A \equiv I_n$ we have $\mu_{z_0} \equiv 1$ for every $z_0\in \overline \Om$. From the assumption \eqref{hypA2} on $A$ one easily checks that
\begin{equation}\label{v0}
\lambda \leq \mu_{z_0}(x) \leq \lambda^{-1},\ \ \ \ \ \ \ x\in \Om.
\end{equation}

We have the following simple lemma whose proof we omit since it is similar to that of Lemma \ref{L:div(Agradr)}.
\begin{lemma}\label{L:mu}
Suppose that $A(z_0) = I_n$.
Then, one has
\begin{itemize}
\item[(1)] $\mu_{z_0}(z_0)=1$,
\item[(2)] $|1-\mu_{z_0}(x)|\le  C|x-z_0|,$
\item[(3)] $|D \mu_{z_0}|\le C$,
\end{itemize}
where $C>0$ is universal.
\end{lemma}

We now introduce a vector field which plays a special role in what follows. With $\mu_{z_0}$ as above, we define 
\begin{equation}\label{Z}
Z_{z_0}(x) = r_{z_0}(x) \frac{A(x) D r_{z_0}}{\mu_{z_0}(x)}=\frac{A(x)(x-z_0)}{\mu_{z_0}(x)}.
\end{equation}
A crucial property of $Z_{z_0}$ is that, denoting by $\nu$ the outer unit normal to the sphere  $\partial B_r(z_0)$, we have 
\begin{equation}\label{ZSr}
<Z_{z_0},\nu>= r_{z_0} \frac{<A(x) D r_{z_0},D r_{z_0}>}{\mu_{z_0}} \equiv r,\ \ \ \ \text{on}\  \partial B_r(z_0).
\end{equation}

Another important fact concerning the vector field $Z_{z_0}$ is contained in the following

\begin{lemma}\label{L:ZZ} 
Suppose that $A(z_0) = I_n$. There exists a universal $O(r_{z_0})$ such that for every $i, j = 1,...,n$, one has
\begin{equation}\label{DZ}
D_i Z_{z_0,j} = \delta_{ij} + O(r_{z_0}).
\end{equation}
In particular, one has
\begin{equation}\label{divZ}
\operatorname{div} Z_{z_0}=n+O(r_{z_0}). 
\end{equation}
\end{lemma}

\begin{proof}
From \eqref{B}, \eqref{Z}, and from 2) and 3) of Lemma  \ref{L:div(Agradr)} we have for a universal $O(r_{z_0})$ 
\begin{align*}
D_i Z_{z_0,j} & = D_i\left(\frac{a_{jk} (x_k - z_{0,k})}{\mu}\right) = \frac{D_i a_{jk} (x_k - z_{0,k})}{\mu} + \frac{a_{jk} \delta_{ki}}{\mu} - \frac{a_{jk} (x_k - z_{0,k})D_i\mu}{\mu^2}
\\
& = \frac{a_{ij}}{\mu} + O(r_{z_0}) =  \frac{\delta_{ij}}{\mu} + O(r_{z_0}) = \delta_{ij} +\delta_{ij} \left(\frac{1}{\mu} - 1\right) + O(r_{z_0}) = \delta_{ij} +  O(r_{z_0}).
\end{align*} 
\end{proof}

When $z_0 = 0$ we simply write $\mu(x)$, and $Z(x)$, instead of $\mu_0(x)$ and $Z_0(x)$.  After these preliminaries, under the hypothesis of Theorem \ref{main} we consider a weak solution of the equation \eqref{m1}. 

\begin{dfn}\label{D:I}
For $z_0\in \overline \Om$ and $r>0$ we define the generalized \emph{height} of $u$ in the ball $B_r(z_0)$ as 
\begin{equation}\label{vh}
H_{z_0}(r)= \int_{\Om \cap B_r(z_0)} u^2  (r^2 - r_{z_0}(x)^2)^{\alpha} \mu_{z_0} = \int_{\Om \cap B_r(z_0)} u^2 (r^2 - |x-z_0|^2)^{\alpha} \mu_{z_0},
\end{equation}
where $\alpha>-1$ is to be fixed later. The generalized \emph{energy} of $u$ in $B_r(z_0)$ is defined by 
\begin{align}\label{vi}
I_{z_0} (r) & = \int_{\Om \cap B_r(z_0)}   <A Du, Du> (r^2 - |x-z_0|^2)^{\alpha+1}  + \int_{\Om \cap B_r(z_0)}   Vu^2 (r^2 - |x-z_0|^2)^{\alpha+1}.
\end{align}
The generalized \emph{frequency} of $u$ in $B_r(z_0)$ is given by
\begin{equation}\label{genf}
N_{z_0}(r) = \frac{I_{z_0} (r)}{H_{z_0}(r)}.
\end{equation}
When $z_0 = 0$ we agree to simply write $H(r), I(r)$ and $N(r)$, instead of $H_0(r), I_0(r)$, $N_0(r)$.
\end{dfn}

Before proceeding we make the observation (important for the computations in this section) that, thanks to Theorem 5.5 in \cite{Li}, under the assumptions of Theorem \ref{main} above, we know that the weak solution $u$ of \eqref{m1} is in $C^1(\overline \Om \cap B_{1}(x_0))$. Therefore in the ensuing computations all derivatives are classical.

In Lemmas \ref{L:Ialt}, \ref{L:fvH}  and Proposition \ref{P:fvegen} below, by translation, we can without loss of generality assume that  $z_0=0$. We stress that $z_0$ need not necessarily be a point on $\partial \Om$. When for a given $r>0$ we have $\partial \Om \cap B_r(z_0)\not= \varnothing$, then in the integrations by parts we will eliminate integrals on such portion of the boundary of $\Om \cap B_r(z_0)$ by using the assumption $u = 0$ on $\Gamma$ in Theorem \ref{main}. If instead $\partial \Om \cap B_r(z_0) = \varnothing$, then $\Om \cap B_r(z_0) = B_r(z_0)$, and corresponding integrals on $\partial B_r(z_0)$ will be eliminated by the weight $(r^2 - |x-z_0|^2)^{\alpha}$ in \eqref{vh} and \eqref{vi} above. 

The following hypothesis 
\begin{equation}\label{dev1}
A(z_0) = A(0)= I_n 
\end{equation}
will be tacitly assumed as in force in Lemmas \ref{L:Ialt}, \ref{L:fvH}  and Proposition \ref{P:fvegen}. 
We will need the following alternative expression of the generalized energy $I(r)$.

\begin{lemma}\label{L:Ialt}
For every $r\in (0,1)$ one has 
\begin{align}\label{alt}
I(r) = & 2 (\alpha +1) \int_{\Om \cap B_r} u <ADu,x> (r^2 - |x|^2)^{\alpha} 
 = 2 (\alpha +1) \int_{\Om \cap B_r} u\ Z u\ (r^2 - |x|^2)^{\alpha}\  \mu.
\end{align}
\end{lemma}

\begin{proof}
From \eqref{vi} and the divergence theorem we obtain
\begin{align}\label{I2}
I(r) & =  \int_{\Om \cap B_r}   <A Du, Du> (r^2 - |x|^2)^{\alpha+1}  + \int_{\Om \cap B_r}   Vu^2 (r^2 - |x|^2)^{\alpha+1},
\\
& = \frac 12 \int_{\Om \cap B_r}  \operatorname{div}(A D(u^2)) (r^2- |x|^2)^{\alpha+1}
= - \frac 12 \int_{\Om \cap B_r}  <A D(u^2), D(r^2- |x|^2)^{\alpha+1}>,
\notag\\
& = 2(\alpha+1) \int_{\Om \cap B_r} u <A Du,x> (r^2- |x|^2)^{\alpha},
\notag
\end{align}
where in the second equality we have used the equation \eqref{m1}. The conclusion of the proof now follows by observing that \eqref{Z} gives
\begin{equation}\label{Zu}
Zu\ \mu = <ADu,x>.
\end{equation}

\end{proof}

\begin{lemma}[First variation of the height]\label{L:fvH}
There exist a universal $O(1)$ such that for every $0<r<1$ one has
\begin{equation}\label{vn}
H'(r)= \frac{2 \alpha+ n }{r} H(r)   + O(1) H(r)+ \frac{1}{(\alpha +1)r} I(r).
\end{equation}
\end{lemma}

\begin{proof}
Differentiating \eqref{vh} we have 
\begin{equation}\label{v2}
H'(r)= 2 \alpha r \int_{\Om \cap B_r} u^2  (r^2 - |x|^2)^{\alpha-1} \mu = \frac{2 \alpha}{r} H(r) + \frac{2 \alpha}{r} \int _{\Om \cap B_r} u^2 (r^2 - |x|^2)^{\alpha-1} |x|^2 \mu ,
\end{equation}
where in the second equality we have used the simple fact
\[
(r^2 - |x|^2)^{\alpha-1} = \frac{1}{r^2} (r^2 - |x|^2)^{\alpha} + \frac{|x|^2}{r^2} (r^2 - |x|^2)^{\alpha-1}.
\]
Using the definitions \eqref{mu} and \eqref{Z} of $\mu$ and $Z$, we see that the second term in the right-hand side of \eqref{v2} equals
\begin{equation}\label{ibp1}
\frac{2 \alpha}{r} \int _{\Om \cap B_r} u^2 (r^2 - |x|^2)^{\alpha-1} |x|^2  \mu  = - \frac{1}{r} \int_{\Om \cap B_r} u^2 <Z, D (r^2 - |x|^2)^\alpha> \mu. 
\end{equation}
Now we notice that since we are assuming that $u$ vanishes continuously on the open subset $\Gamma$ of $\partial \Om$, there exists $R_1 = R_1(z_0)>0$ such that $u$ vanishes on $\partial \Om \cap B_r$ for every $0<r<R_1$. Applying the divergence theorem to the right-hand side of \eqref{ibp1}, we thus find
\begin{align*}
& \frac{2 \alpha}{r} \int _{\Om \cap B_r} u^2  (r^2 - |x|^2)^{\alpha-1} |x|^2 \mu
 =  \frac{1}{r} \int_{\Om \cap B_r} \operatorname{div}(u^2 \mu Z) (r^2 - |x|^2)^\alpha
\\
& = \frac{1}{r} \int_{\Om \cap B_r} (\operatorname{div} Z) u^2 (r^2 - |x|^2)^\alpha \mu  + \frac{2}{r} \int_{\Om \cap B_r} u Zu  (r^2 - |x|^2)^\alpha \mu
\\
& + \frac{1}{r} \int_{\Om \cap B_r} u^2 Z\mu (r^2 - |x|^2)^\alpha.
\end{align*}

We can thus apply Lemma \ref{L:ZZ} that allows to conclude for a universal $O(r)$
\begin{align*}
& \frac{2 \alpha}{r} \int _{\Om \cap B_r} u^2  (r^2 - |x|^2)^{\alpha-1} |x|^2 \mu
 = \frac{n + O(r)}{r} \int_{\Om \cap B_r}  u^2  (r^2 - |x|^2)^\alpha \mu
 \\
 & + \frac{2}{r} \int_{\Om \cap B_r} u Zu  (r^2 - |x|^2)^\alpha \mu
+ \frac{1}{r} \int_{\Om \cap B_r} u^2 Z\mu (r^2 - |x|^2)^\alpha.
\end{align*}
Furthermore, from \eqref{v0} and (3) of Lemma \ref{L:div(Agradr)} we find
\[
Z\mu = \frac{<Ax,D\mu>}{\mu} = O(r),
\]
with $O(r)$ universal. Now, \eqref{Zu} above  gives
\begin{align*}
& \frac{2 \alpha}{r} \int _{\Om \cap B_r} u^2  (r^2 - |x|^2)^{\alpha-1} |x|^2 \mu
 = \frac{n}{r} H(r) + O(1) H(r) + \frac{2}{r} \int_{\Om \cap B_r} u <ADu,x> (r^2 - |x|^2)^\alpha.
\end{align*}
Substituting this identity in \eqref{v2}, we conclude that 
\begin{equation}\label{H2}
H'(r)= \frac{2 \alpha+ n }{r} H(r)   + O(1) H(r) + \frac{2}{r} \int_{\Om \cap B_r} u  < ADu, x> (r^2-|x|^2)^\alpha.
\end{equation}
From \eqref{H2} and \eqref{vi} we obtain the desired conclusion \eqref{vn} above.

\end{proof}

In the next result we will employ a geometric notion which has already been introduced in Lemma \ref{L:gens} above.

\begin{prop}[First variation of the energy]\label{P:fvegen}
 Let $u$ be a solution to \eqref{m1} and assume that for some $R_1> 0$ the set   $\Om \cap B_{R_1}$ be generalized  star-shaped with respect to $A$ and $z_0 = 0$ in the sense that
\begin{equation}\label{sr}
<Z(x),\nu(x)>  \geq 0,
\end{equation}
where $\nu(x)$ is the outward unit normal in $x \in \partial \Om \cap B_{R_1}$.
If $u$ vanishes on $\partial \Om \cap B_{R_1}$, then there exist $O(1)$ and $C>0$ universal, but independent from $M \ge ||V||_{W^{1,\infty}(\Om)}$, such that for every $0<r<R_1$ one has
\begin{align}\label{gfve} 
I'(r) & \ge \left(\frac{2\alpha+n}{ r} + O(1)\right) I(r) - C M r  H(r) 
 +  \frac{4(\alpha+1)}{r} \int_{\Om \cap B_r} (Zu)^2  (r^2 - |x|^2)^{\alpha}\mu.
\end{align}
\end{prop}

\begin{proof}
From the identity \eqref{I2} above we obtain
\[
I'(r)  = 2(\alpha+1) r \int_{\Om \cap B_r} <A Du, Du> (r^2 - |x|^2)^{\alpha} + 2(\alpha +1) r \int_{\Om \cap B_r} Vu^2 (r^2 - |x|^2)^{\alpha}.
\]
Using the trivial observation that 
\[
<Z,x> = \frac{<Ax,x>}{\mu} = |x|^2,
\]
 in the first term in the right-hand side of the above equation for $I'(r)$ we now use the fact that
\begin{align*}
(r^2 - |x|^2)^{\alpha} & = \frac{1}{r^2} (r^2 - |x|^2)^{\alpha+1} + \frac{1}{r^2} (r^2 - |x|^2)^{\alpha} <Z,x> 
\\
& = \frac{1}{r^2} (r^2 - |x|^2)^{\alpha+1} - \frac{1}{2(\alpha+1)r^2} <Z,D(r^2 - |x|^2)^{\alpha+1}>.
\end{align*}
We thus find
\begin{align*}
I'(r) & = \frac{2(\alpha+1)}{ r} \int_{\Om \cap B_r} <A Du, Du> (r^2 - |x|^2)^{\alpha+1} + 2(\alpha +1) r \int_{\Om \cap B_r} Vu^2 (r^2 - |x|^2)^{\alpha}.
\\
& - \frac{1}{r} \int_{\Om \cap B_r} <ADu, Du> <Z,D(r^2 - |x|^2)^{\alpha+1}>.
\end{align*}
We now integrate by parts in the third term in the right-hand side of the latter equation obtaining
\begin{align*}
& - \frac{1}{r} \int_{\Om \cap B_r} <ADu, Du> <Z,D(r^2 - |x|^2)^{\alpha+1}> 
\\
& = \frac{1}{r} \int_{\Om \cap B_r} \operatorname{div}(<ADu, Du> Z) (r^2 - |x|^2)^{\alpha+1}
\\
&  - \frac{1}{r} \int_{\partial \Om \cap B_r} <ADu, Du> <Z, \nu>  (r^2 - |x|^2)^{\alpha+1}.
\end{align*}

To compute the first integral in the right-hand side of the latter equation we use the following generalization of the classical identity of Rellich due to Payne and Weinberger, see \cite{PW}, and also section 5.1 in \cite{N},
\begin{align*}
\operatorname{div}(<ADu, Du>Z)& =  2 \operatorname{div}(<Z,Du> A Du)  + \operatorname{div} Z <ADu, Du> 
\\
& - 2 D_i Z_k a_{ij} D_j u D_k u  - 2 <Z,Du> \operatorname{div} (A Du) + Z_k D_k a_{ij} D_i u D_j u.
\end{align*}
Using \eqref{m1}, \eqref{v0}, \eqref{Z}, and \eqref{DZ} and \eqref{divZ} in Lemma \ref{L:ZZ} in the latter equation,
we obtain
\begin{align*}
\operatorname{div}(<ADu, Du>Z)& =  2 \operatorname{div}(<Z,Du> A Du)  + ((n-2) + O(r)) <ADu, Du>  - 2 <Z,Du> V u.
\end{align*}
From this equation and the divergence theorem we find 
\begin{align*}
& \frac{1}{r} \int_{\Om \cap B_r} \operatorname{div}(<ADu, Du> Z) (r^2 - |x|^2)^{\alpha+1}
 =  \frac{2}{r} \int_{\partial \Om \cap B_r} Zu <ADu,\nu> (r^2 - |x|^2)^{\alpha+1}
\\
& + \left(\frac{n-2}{r} + O(1)\right)\int_{\Om \cap B_r} <ADu, Du> (r^2 - |x|^2)^{\alpha+1}  -  \frac{1}{r} \int_{\Om \cap B_r} Z(u^2) V (r^2 - |x|^2)^{\alpha+1}
\\
& +  \frac{4(\alpha+1)}{r} \int_{\Om \cap B_r} (Zu)^2 \mu (r^2 - |x|^2)^{\alpha}.
\end{align*}
Using these equations we thus find
\begin{align*}
I'(r) & = \left(\frac{2\alpha+n}{ r} + O(1)\right)\int_{\Om \cap B_r} <A Du, Du> (r^2 - |x|^2)^{\alpha+1} + 2(\alpha +1) r \int_{\Om \cap B_r} Vu^2 (r^2 - |x|^2)^{\alpha}.
\\
& -  \frac{1}{r} \int_{\Om \cap B_r} Z(u^2) V (r^2 - |x|^2)^{\alpha+1} +  \frac{4(\alpha+1)}{r} \int_{\Om \cap B_r} (Zu)^2 \mu (r^2 - |x|^2)^{\alpha} 
\\
& + \frac{2}{r} \int_{\partial \Om \cap B_r} Zu <ADu,\nu> (r^2 - |x|^2)^{\alpha+1} - \frac{1}{r} \int_{\partial \Om \cap B_r} <ADu, Du> <Z, \nu>  (r^2 - |x|^2)^{\alpha+1}.
\end{align*}
Observe now that, since $u = 0$ on $\Gamma$, we have $Du(x) = \beta(x) \nu(x)$ for a certain function $\beta$. Therefore, we have on $\Gamma$
\[
Zu <ADu,\nu> = <ADu, Du> <Z, \nu>.
\] 
If we use the generalized starlikeness assumption \eqref{sr} above, we thus infer
\begin{align}\label{bdc}
& \frac{2}{r} \int_{\partial \Om \cap B_r} Zu <ADu,\nu> (r^2 - |x|^2)^{\alpha+1} - \frac{1}{r} \int_{\partial \Om \cap B_r} <ADu, Du> <Z, \nu>  (r^2 - |x|^2)^{\alpha+1} 
\\
& = \frac{1}{r} \int_{\partial \Om \cap B_r} <ADu, Du> <Z, \nu>  (r^2 - |x|^2)^{\alpha+1} \ge 0.
\notag
\end{align}
This gives
\begin{align*}
I'(r) & \ge  \left(\frac{2\alpha+n}{ r} + O(1)\right)\int_{\Om \cap B_r} <A Du, Du> (r^2 - |x|^2)^{\alpha+1} + 2(\alpha +1) r \int_{\Om \cap B_r} Vu^2 (r^2 - |x|^2)^{\alpha}
\\
& +  \frac{4(\alpha+1)}{r} \int_{\Om \cap B_r} (Zu)^2 (r^2 - |x|^2)^{\alpha} \mu -  \frac{1}{r} \int_{\Om \cap B_r} Z(u^2) V (r^2 - |x|^2)^{\alpha+1}
\\
& = \left(\frac{2\alpha+n}{ r} + O(1)\right) I(r) - \left(\frac{2\alpha+n}{ r} + O(1)\right)\int_{\Om \cap B_r} V u^2 (r^2 - |x|^2)^{\alpha+1} 
\\
& + 2(\alpha +1) r \int_{\Om \cap B_r} Vu^2 (r^2 - |x|^2)^{\alpha} +  \frac{4(\alpha+1)}{r} \int_{\Om \cap B_r} (Zu)^2  (r^2 - |x|^2)^{\alpha} \mu
\\
& -  \frac{1}{r} \int_{\Om \cap B_r} Z(u^2) V (r^2 - |x|^2)^{\alpha+1}.
\end{align*}
An integration by parts now gives
\begin{align*}
& -  \frac{1}{r} \int_{\Om \cap B_r} Z(u^2) V (r^2 - |x|^2)^{\alpha+1} = \frac{1}{r} \int_{\Om \cap B_r} u^2 \operatorname{div}(V (r^2 - |x|^2)^{\alpha+1} Z)
\\
& = \frac{1}{r} \int_{\Om \cap B_r} V u^2  (r^2 - |x|^2)^{\alpha+1} \operatorname{div} Z + \frac{1}{r} \int_{\Om \cap B_r} u^2 ZV (r^2 - |x|^2)^{\alpha+1} 
\\
&  -  \frac{2(\alpha+1)}{r} \int_{\Om \cap B_r} V u^2  (r^2 - |x|^2)^{\alpha} <Z,x> 
\\
& = \frac{n}{r}\int_{\Om \cap B_r} V u^2  (r^2 - |x|^2)^{\alpha+1} + \frac{1}{r}\int_{\Om \cap B_r} O(|x|) V u^2 (r^2 - |x|^2)^{\alpha+1}
\\
& + \frac{1}{r} \int_{\Om \cap B_r} u^2 ZV (r^2 - |x|^2)^{\alpha+1}  -  \frac{2(\alpha+1)}{r} \int_{\Om \cap B_r} V u^2  (r^2 - |x|^2)^{\alpha} |x|^2.
\end{align*}
Substitution in the above inequality gives
\begin{align*}
I'(r) & \ge  \left(\frac{2\alpha+n}{ r} + O(1)\right) I(r) - \left(\frac{2\alpha+n}{ r} + O(1)\right)\int_{\Om \cap B_r} V u^2 (r^2 - |x|^2)^{\alpha+1} 
\\
& + 2(\alpha +1) r \int_{\Om \cap B_r} Vu^2 (r^2 - |x|^2)^{\alpha} +  \frac{4(\alpha+1)}{r} \int_{\Om \cap B_r} (Zu)^2 \mu (r^2 - |x|^2)^{\alpha} 
\\
&  + \frac{n}{r}\int_{\Om \cap B_r} V u^2  (r^2 - |x|^2)^{\alpha+1} + \frac{1}{r}\int_{\Om \cap B_r} O(|x|) V u^2 (r^2 - |x|^2)^{\alpha+1}
\\
& + \frac{1}{r} \int_{\Om \cap B_r} u^2 ZV (r^2 - |x|^2)^{\alpha+1}  -  \frac{2(\alpha+1)}{r} \int_{\Om \cap B_r} V u^2  (r^2 - |x|^2)^{\alpha} |x|^2.
\end{align*}
If we now observe that
\begin{align*}
& 2(\alpha +1) r \int_{\Om \cap B_r} Vu^2 (r^2 - |x|^2)^{\alpha} = \frac{2(\alpha +1)}{ r} \int_{\Om \cap B_r} Vu^2 (r^2 - |x|^2)^{\alpha+1}
\\
& + \frac{2(\alpha +1)}{ r} \int_{\Om \cap B_r} Vu^2 (r^2 - |x|^2)^{\alpha} |x|^2,
\end{align*}
 then the previous inequality gives
\begin{align*}
I'(r) & \ge  \left(\frac{2\alpha+n}{ r} + O(1)\right) I(r) + \left(\frac{2}{ r} + O(1)\right)\int_{\Om \cap B_r} V u^2 (r^2 - |x|^2)^{\alpha+1} 
\\
&  +  \frac{4(\alpha+1)}{r} \int_{\Om \cap B_r} (Zu)^2 \mu (r^2 - |x|^2)^{\alpha} 
+ \frac{1}{r}\int_{\Om \cap B_r} O(|x|) V u^2 (r^2 - |x|^2)^{\alpha+1}
\\
& + \frac{1}{r} \int_{\Om \cap B_r} u^2 ZV (r^2 - |x|^2)^{\alpha+1}.
\end{align*} 
It is now clear that  
\[
\left|\int_{\Om \cap B_r} V u^2 (r^2 - |x|^2)^{\alpha+1} \right| \le r^2 ||V||_{L^\infty(\Om)} \int_{\Om \cap B_r} u^2 (r^2 - |x|^2)^{\alpha} \le C r^2 ||V||_{W^{1,\infty}(\Om)} H(r)
\]
where in the last equality $C>0$ is universal and we have used \eqref{v0} and \eqref{vh} above. Similarly, we have
\[
\left|\int_{\Om \cap B_r} O(|x|) V u^2 (r^2 - |x|^2)^{\alpha+1}\right| \le C r^3 ||V||_{W^{1,\infty}(\Om)} H(r),
\]
with $C>0$ universal. Finally, we have 
\[
\left|\int_{\Om \cap B_r} u^2 ZV (r^2 - |x|^2)^{\alpha+1}\right| \le C r^3 ||V||_{W^{1,\infty}(\Om)} H(r),
\]
with $C>0$ universal. These estimates allow to conclude that the desired inequality \eqref{gfve} does hold.

\end{proof}

The following important consequence of Lemma \ref{L:fvH} and Proposition \ref{P:fvegen} is the central result of this section. 

\begin{thrm}[Monotonicity of the generalized frequency]\label{vimp}
 Let $u$ be a solution to \eqref{m1} and assume that for $z_0 = 0 \in \overline{\Om}$ the assumption \eqref{dev1} hold.  Suppose that for  some $ R_1> 0$ the set  $\Om \cap B_{R_1}$ satisfy the generalized  star-shaped assumption \eqref{sr} above with respect to $0$.
If $u$ vanishes on $\partial \Om \cap B_{R_1}$, then there exist $\overline{R}, C_1, C_2>0$, depending on $n, \lambda, K$, but not on $M$, such that  the function
\[
r \to e^{C_1r} (N(r) + C_2 M r^2)
\]
is nondecreasing for $ 0< r< \min\{\overline{R}, R_1\}$.
\end{thrm}

\begin{proof}
From \eqref{genf}, Lemma \ref{L:fvH} and Proposition \ref{P:fvegen}  we have for $0<r<R_1$
\begin{align}\label{fn'}
N'(r) & = \frac{I'(r)}{H(r)} - \frac{H'(r)}{H(r)} N(r) \ge \left(\frac{2\alpha + n}{r} + O(1)\right) N(r) - C Mr 
\\
& +  \frac{4(\alpha+1)}{r H(r)} \int_{\Om \cap B_r} (Zu)^2 \mu (r^2 - |x|^2)^{\alpha}
 -  \left(\frac{2\alpha + n}{r} + O(1) + \frac{1}{(\alpha+1)r} N(r)\right) N(r)
\notag \\
 & = O(1) N(r) - C M r  +  \frac{4(\alpha+1)}{r H(r)} \int_{\Om \cap B_r} (Zu)^2 \mu (r^2 - |x|^2)^{\alpha} -  \frac{1}{(\alpha+1)r} N(r)^2
\notag \\
 & \ge - C_1 N(r) - C M r,
\notag
\end{align}
where the last inequality follows from \eqref{vi} above and the Cauchy-Schwarz inequality, and $C_1, C$ are universal.
Letting $C_2 =  C/2$ we now conclude
\begin{align*}
\frac{d}{dr} e^{C_1 r} (N(r) + C_2 M r^2) & = e^{C_1 r} \left(N'(r) + C_1 N(r) + C M r + C_1 C_2 M r^2\right)
\\
&  \ge N'(r) + C_1 N(r) + C M r \ge 0,
\end{align*}
where the last inequality follows from \eqref{fn'}.

\end{proof}

\section{Some three-sphere lemmas}\label{S:3s}
 
The aim of this section is to derive some basic consequences of the monotonicity Theorem \ref{vimp} above. We begin with establishing a three-sphere theorem for the height function $H$, Lemma \ref{L:3sH} below, and then combine such result with local estimates at the boundary to obtain a corresponding three-sphere theorem for $L^\infty$ norms on balls, see Lemma \ref{L:3sh} below.

\begin{lemma}\label{L:3sH}
Under the hypothesis of  Theorem \ref{vimp}, let $0<r_1 < r_2 < 2r_2 < r_3< R_1$. Then, there exist universal constants $\overline C, C$ and $C'$ such that, letting
\[
\alpha_0 = 
 \log\left(\frac{r_3}{2r_2}\right),\ \ \ \beta_0 = \overline C^{2} \log \left(\frac{2r_2}{r_1}\right),
 \]
we obtain
\begin{equation}\label{warning4}
H(2r_2) \le e^C \left(\frac{r_3}{2r_2}\right)^{C' \sqrt M} H(r_3)^\frac{\beta_0}{\alpha_0 + \beta_0} H(r_1)^\frac{\alpha_0}{\alpha_0 + \beta_0}.
\end{equation} 
\end{lemma}

\begin{proof}
Returning to \eqref{vn}, we rewrite it in the following form
\begin{equation}\label{vnbis}
\frac{d}{dr} \log\left(\frac{H(r)}{r^{2 \alpha+ n}}\right) = O(1) + \frac{1}{(\alpha +1)r} N(r),\ \ \ \ 0<r<R_1,
\end{equation}
where $|O(1)| \le C$, with $C$ universal. Without loss of generality we assume that $R_1\le 1$. From Theorem \ref{vimp} we have
\[
e^{C_1r} (N(r) + C_2 M r^2) \le e^{C_1\rho} (N(s) + C_2 M s^2), \ \ \ \ \ \ \ \ \ \text{for}\  0<r<s<R_1.
\]
The latter monotonicity property implies, in particular, the existence of universal constants $C_2>0$ and $\overline C>0$, such that
\begin{equation}\label{ineq}
N(r) \leq \overline C (N(s) + C_2 M),\ \ \ \ \ \ \ \ \ \text{for}\  0<r<s<R_1.
\end{equation}
Without of loss of generality we assume $\overline C \ge 1$. Suppose now that $0<r_1 < r_2 < 2r_2 < r_3< R_1$.  
Integrating \eqref{vnbis} between $r_1$ and $2r_2$, and using \eqref{ineq}, we find 
\begin{equation}\label{top}
\frac{\log \frac{H(2r_2)}{H(r_1)} - C}{\log \left(\frac{2r_2}{r_1}\right)} - (2\alpha + n)  \le   \frac{\overline C}{\alpha + 1} \left(N(2r_2) + C_2 M\right).
\end{equation}
Next, we integrate \eqref{vnbis} between $2r_2$ and $r_3$, and again using \eqref{ineq} we find 
\begin{equation}\label{bottom}
\frac{\overline C}{\alpha + 1} \left(N(2r_2) - \overline C C_2 M \right)
\le  \overline{C}^2 \left[\frac{\log \frac{H(r_3)}{H(2r_2)} + C}{
 \log\left(\frac{r_3}{2r_2}\right)}  - (2\alpha + n)\right].
\end{equation}
Combining \eqref{top} and \eqref{bottom} we conclude
\[
\frac{\log \frac{H(2r_2)}{H(r_1)} - C}{\overline{C}^2 \log \left(\frac{2r_2}{r_1}\right)}   \le  \frac{\log \frac{H(r_3)}{H(2r_2)} + C}{
 \log\left(\frac{r_3}{2r_2}\right)} + C' \frac{M}{\alpha +1} - \left(1 - \frac{1}{\overline{C}^2}\right)(2\alpha + n),
 \]
 where we have let $C' = (\overline C + 1)/\overline C$. Since $\overline C \ge 1$, if we now set
\[
\alpha_0 = 
 \log\left(\frac{r_3}{2r_2}\right),\ \ \ \beta_0 = \overline C^{2} \log \left(\frac{2r_2}{r_1}\right),
 \]
then  we obtain
\begin{equation}\label{warning}
\alpha_0 \log \frac{H(2r_2)}{H(r_1)} \le \beta_0  \log \frac{H(r_3)}{H(2r_2)}  + C (\alpha_0 + \beta_0) + C' \frac{M}{\alpha + 1} \alpha_0 \beta_0.
\end{equation}
Dividing both sides of the latter inequality by the quantity $\alpha_0 + \beta_0$, we find
\[
 \log \left(\frac{H(2r_2)}{H(r_1)}\right)^\frac{\alpha_0}{\alpha_0 + \beta_0} \le   \log \left(\frac{H(r_3)}{H(2r_2)}\right)^\frac{\beta_0}{\alpha_0 + \beta_0} + C  + C' \frac{M}{\alpha + 1} \frac{\alpha \beta_0}{\alpha_0 + \beta_0}.
\]
This gives
\begin{equation}\label{warning2}
\log H(2r_2) \le \log \left[H(r_3)^\frac{\beta_0}{\alpha_0 + \beta_0} H(r_1)^\frac{\alpha_0}{\alpha_0 + \beta_0}\right] + C  + C' \frac{M}{\alpha + 1} \alpha_0,
\end{equation}
 where we have used the trivial estimate $\frac{\beta_0}{\alpha_0 + \beta_0} \le 1$. Exponentiating both sides of \eqref{warning2}  and letting $\alpha = \sqrt M$, we reach the desired conclusion \eqref{warning4}.

\end{proof}

Lemma \ref{L:3sH} implies the following three-sphere theorem for the $L^\infty$ norms.

\begin{lemma}\label{L:3sh}
Under the hypothesis of  Theorem \ref{vimp}, let $0<r_1 < r_2 < 2r_2 < r_3< R_1$. Then, there exist universal constants $\overline C, C, C^\star$ and $C'$ such that, letting
\[
\alpha_1 = 
 \log\left(\frac{r_3}{2(r_2 + r_3)/3}\right),\ \ \ \beta_1 = \overline C^{2} \log \left(\frac{2(r_2 + r_3)/3}{r_1}\right),
\]
we obtain
\begin{align}\label{3ball2}
& ||u||_{L^{\infty}(\Om \cap B_{r_2})}  \leq C e^{C^\star \sqrt{M}} \left(\frac{r_3}{r_3 - 2r_2}\right)^{\frac n2}  \left(\frac{r_3}{2(r_2+r_3)/3}\right)^{C''\sqrt{M}} 
\\
& \times  ||u||_{L^\infty(\Om\cap B_{r_3}})^\frac{\beta_1}{\alpha_1 + \beta_1} ||u||_{L^\infty(\Om\cap B_{r_1}})^\frac{\alpha_1}{\alpha_1 + \beta_1}.
\notag
\end{align}
\end{lemma}

\begin{proof}
We introduce the quantity
\begin{equation}\label{geh}
h(r)= \int_{\Om \cap B_r} u^2 \mu.
\end{equation}
One has trivially 
\[
H(r) \leq r^{2\alpha} h(r),\ \ \ \ \  \text{and}\ \ \ \ \ h(r) \leq \frac{H(\rho)}{(\rho^2 - r^2)^{\alpha}},\  0 < r < \rho < R_1.
\]
Using such estimates in \eqref{warning4} we arrive at
\begin{equation}\label{3ball1}
h(r_2) \le e^C  (\frac{r_3}{2r_2})^{C''\sqrt{M}}  h(r_3)^\frac{\beta_0}{\alpha_0 + \beta_0} h(r_1)^\frac{\alpha_0}{\alpha_0 + \beta_0},
\end{equation}
with the universal constant $C'' = C' + 2$. 

Since $u$ vanishes continuously on  $\Gamma \subset \partial \Om$ by classical boundary estimates there exists $C = C(n,\lambda)>0$  such that for any $x_0\in \Gamma$ and $0<r<\rho< R_1$ one has
\begin{align}\label{best}
\underset{\Om \cap B(x_0,r)}{\sup}\ |u| \le \frac{C (1+ ||V||_{L^\infty(\Om)})^{\frac n2}}{(\rho-r)^{\frac n2}} \left(\int_{\Om \cap B(x_0,\rho)} u^2\right)^{\frac 12}.
\end{align}
The estimate \eqref{best} can be established as follows. First, we locally flatten the boundary of $\Om$ obtaining an equation of the type \eqref{m1} above, in which the principal part is a uniformly elliptic operator with $C^{0,\operatorname{Dini}}$ coefficients. Secondly, we perform an odd reflection to reduce the above estimate to an interior one for a variable coefficient operator in which now the coefficients are just bounded measurable. We can then invoke Theorem 8.17 in \cite{GT} with $p=2$ to conclude the above inequality \eqref{best}. 
From \eqref{best} and \eqref{v0} above, we immediately obtain for any  $0<r<\rho< R_1$
\begin{equation}\label{lle}
||u||_{L^{\infty}(\Om \cap B_{r})} \le  \frac{C (1+ ||V||_{L^\infty(\Om)})^{\frac n2}}{(\rho-r)^{\frac n2}} h(\rho)^{\frac 12}.
\end{equation}
If we now use \eqref{lle} with $r = r_2$ and $\rho = (r_3 + r_2)/3$, we obtain
\[
||u||^2_{L^{\infty}(\Om \cap B_{r_2})} \leq \frac{C}{(r_3 - 2r_2)^n} C(1+ ||V||_{L^\infty(\Om)})^n h((r_2 + r_3)/3)
\]
Since $r_1 <(r_2 + r_3)/3 < 2(r_2 + r_3)/3 < r_3$, we can apply \eqref{3ball1}, with $r_2$ replaced by
$(r_2 + r_3)/3$, obtaining
\[
h((r_2 + r_3)/3) \le e^C  (\frac{r_3}{2(r_2+r_3)3})^{C''\sqrt{M}}  h(r_3)^\frac{\beta_1}{\alpha_1 + \beta_1} h(r_1)^\frac{\alpha_1}{\alpha_1 + \beta_1},
\]
where 
\[
\alpha_1 = 
 \log\left(\frac{r_3}{2(r_2 + r_3)/3}\right),\ \ \ \beta_1 = \overline C^{2} \log \left(\frac{2(r_2 + r_3)/3}{r_1}\right).
\]
Combining the last two inequalities we find 
\[
||u||^2_{L^{\infty}(\Om \cap B_{r_2})} \leq C(1+ ||V||_{L^\infty(\Om)})^n  \frac{1}{(r_3 - 2r_2)^n}  \left(\frac{r_3}{2(r_2+r_3)/3}\right)^{C''\sqrt{M}}  h(r_3)^\frac{\beta_1}{\alpha_1 + \beta_1} h(r_1)^\frac{\alpha_1}{\alpha_1 + \beta_1}.
\]
Next, from \eqref{v0} we have the trivial estimate
\[
h(r) \le  \lambda^{-1} \omega_n r^n ||u||^2_{L^\infty(\Om\cap B_r)},
\]
where $\omega_n$ is the $n$-dimensional volume of the unit ball in $\Rn$. Together with the previous estimate, this gives \eqref{3ball2},
where $C^\star>0$ is such that $(1+x)^{\frac n2} \le e^{C^\star \sqrt x}$ for every $x\ge 0$. 

\end{proof}

\medskip

Suppose now that $z_0\in \overline \Om$ is a point at which the following holds:
\begin{itemize}
\item[(i)] $A(z_0) =  I_n$; 
\item[(ii)] $\Om \cap B_{r_3}(z_0)$ is generalized star-shaped with respect to $z_0$ 
as in \eqref{dev1} above.
\end{itemize}
Then, arguing as in the proof of \eqref{3ball2}, we obtain:

\begin{align}\label{3ball10}
& ||u||_{L^{\infty}(\Om \cap B_{r_2}(z_0))}  \leq C e^{C^\star \sqrt{M}} \left(\frac{r_3}{r_3 - 2r_2}\right)^{\frac n2}  \left(\frac{r_3}{2(r_2+r_3)/3}\right)^{C''\sqrt{M}} 
\\
& \times  ||u||_{L^\infty(\Om\cap B_{r_3}(z_0)})^\frac{\beta_1}{\alpha_1 + \beta_1} ||u||_{L^\infty(\Om\cap B_{r_1}(z_0)})^\frac{\alpha_1}{\alpha_1 + \beta_1}.
\notag
\end{align}

\medskip
 
\begin{rmrk}\label{R:int}
Before proceeding, we pause to  note that the interior  analogue of Theorem \ref{vimp} continues to be valid, i.e., when $\Om \cap B_r(z_0)=B_r(z_0)$. This follows from the fact that, in such  situation, thanks to the presence of the weight $(r^2 - |x-z_0|^2)^{\alpha}$ in the definitions \eqref{vh} and \eqref{vi}, in the computations leading to \eqref{gfve} in Proposition \ref{P:fvegen} above all the boundary integrals, with the exception of those in \eqref{bdc} above, cancel. However, \eqref{bdc} continues to be valid since, thanks to \eqref{ZSr} above, the whole domain $B_r(z_0)$ is star-shaped in the generalized sense of \eqref{sr} with respect to the matrix-valued function $A(x)$ and the center $z_0$. 
From the interior analogue of Theorem \ref{vimp} we can subsequently deduce  that the interior analogues of \eqref{3ball1}, \eqref{3ball2} and \eqref{3ball10}
 also hold for solutions of $\operatorname{div} (A(x) Du)=Vu$. 
\end{rmrk}

\section{The main growth lemmas}\label{S:gm}

This section is devoted to proving two quantitative growth lemmas which constitute the backbone of the main result in this paper, Theorem \ref{main} above. As it will soon become apparent, the treatment of variable coefficients operators requires a certain amount of technical work. In this respect, the core result of this section is Lemma \ref{imp2} below.

\begin{lemma}\label{L:vegen}
Let $\Om, u, x_0$ be as in Theorem \ref{main}. Then, there exist universal constants $L_1, L_2$ such that for a sufficiently small universal $r_0$ 
\begin{equation}\label{h34}
\ve= \sup_{B_{\frac{r_0}{4}}(x_0) \cap \Om} |u| \geq L_1 \exp ( - L_2 (\sqrt{M} + 1)).
\end{equation}
\end{lemma}

\begin{proof}

By a rotation and translation we can assume that $x_0 = 0$ and $\nu(0) = e_n$. As in Lemma \ref{L:star} above, we consider the interior point $y_0 = - a \nu(0)$ associated with $x_0 = 0$. If $A(y_0) \not= I_n$, we use the change of coordinates $T_{y_0}$ in \eqref{T} of Section \ref{S:moving} above, and indicate with $\Om_{y_0} = T_{y_0}(\Om)$. Having done this, we now have $A_{y_0}(0)=I_n$. We want to apply Lemma \ref{L:3sh} to obtain  \eqref{3ball2} with $\Om$ replaced by $\Om_{y_0}$ and $u$ replaced by $u_{y_0}$. Thanks to Lemma \ref{L:gens} for every $0<r \le R_0$ the set $\Om_{y_0} \cap B_{\sqrt \lambda (r-a)}$ is generalized star-shaped with respect to $0$, therefore the hypothesis of Lemma \ref{L:3sh} are fulfilled by $\Om_{y_0}$ and $A_{y_0}$.    
If we let 
\[
r_1= \lambda^{5/2}( \frac{r}{4}-a),\ \ \ \ r_2= \lambda(\frac{7r}{15}- \frac{a}{4})\ \ \ \ r_3= \lambda(r-a),
\]
 then since $\lambda \le 1$, we trivially have  $r_3 \le \lambda^{1/2}( r - a)$, so that we fall within the range of \eqref{St1}
 in Lemma \ref{L:gens} above, and we also clearly have $0<r_1 < r_2 < 2r_2 < r_3$. We conclude that, with the above choice of $r_1, r_2, r_3$, the estimate \eqref{3ball2} holds with $\Om$ replaced by $\Om_{y_0}$, and $u$ replaced by $u_{y_0}$. Since $|u|\leq C_0$ trivially implies $||u_{y_0}||_{L^{\infty}(B_{r_3} \cap \Om_{y_0})} \leq C_0$, we obtain from \eqref{3ball2} 
\begin{equation}\label{l32}
||u_{y_0}||_{L^{\infty}(\Om_{y_0}\cap B_{3 \lambda r/8}(p'))} \leq  C \exp(C(\sqrt{M}+ 1) ) ||u_{y_0}||_{L^{\infty}(\Om_{y_0} \cap B_{\lambda^2r/4}(p'))}^{\theta_0}.
\end{equation}
where $p'= T_{y_0}(0)$. From \eqref{cont} and \eqref{l32},  by using $T_{y_0}^{-1}$, we find 
\begin{equation}\label{g32}
||u||_{L^{\infty}(\Om \cap B_{3 \lambda^{3/2} r/8})} \leq  C \exp(C(\sqrt{M}+ 1) ) ||u||_{L^{\infty}(\Om \cap B_{\lambda^{3/2} r/4})}^{\theta_0}.
\end{equation}
If in \eqref{g32} we substitute $\lambda^{3/2} r$ with $r$, we obtain for all sufficiently small $r$
\begin{equation}\label{g34}
||u||_{L^{\infty}(\Om \cap B_{3 r/8})} \leq  C \exp(C(\sqrt{M}+ 1) ) ||u||_{L^{\infty}(\Om \cap B_{ r/4})}^{\theta_0}.
\end{equation}
Recalling that the origin is just the image of an arbitrary point $z \in \Gamma \cap B_{3/2}$ after translation and rotation,  we conclude from \eqref{g34} that for any such $z \in \Gamma \cap B_{3/2}$ and $r_0$ sufficiently small, but universal,  we have 
\begin{equation}\label{g35}
||u||_{L^{\infty}(\Om \cap B_{3 r_0/8}(z))} \leq  C \exp(C(\sqrt{M}+ 1) ) ||u||_{L^{\infty}(\Om \cap B_{ r_0/4}(z))}^{\theta_0}.
\end{equation}
In a similar way, the interior analogue  of \eqref{g35} can be established, i.e., when the relevant ball does not intersect $\partial \Om$.  Keeping in mind that we have let $x_0 = 0$, and that, for a suitably fixed $r_0$, we have defined
\[
\ve= \sup_{B_{\frac{r_0}{4}}(x_0) \cap \Om} |u|,
\]
we can re-write \eqref{g34} as follows 
\begin{equation}\label{g34bis}
||u||_{L^{\infty}(\Om \cap B_{3 r_0/8})} \leq  C_1 \exp(C(\sqrt{M}+ 1) )\ \ve^{\theta},
\end{equation}
where $\theta>0$ is universal, and $C_1>0$ is a new constant that also incorporates the $L^\infty$ bounds for $u$ in $\Om \cap B_{r_0}$, controlled in turn by the quantity $C_0$ in Theorem \ref{main}. 

To complete the proof we now argue as follows.
By the assumption in Theorem \ref{main} that  $\sup_{\Om \cap B_{1}} |u| \geq 1$, there  exists $\overline{x} \in \Om \cap B_1$ such that $|u(\overline{x})| = \sup_{\Om \cap B_{1}} |u| \geq 1$. Let
\begin{equation}
d_1= d(\overline{x},\Gamma \cap \overline{B}_{\frac 32}),
\end{equation}
where $d(x, H)= \inf \{|x-h|\mid h \in H\}$. There exist two possibilities.

\medskip

\noindent \textbf{Case $1$:} $d_1 \geq \frac{3r_0}{16}$. In such case, we take a  chain of balls $B_{\ell r_0}(x_i)$, $i =  1,...,d$, where $\ell$ is a sufficiently small constant depending on $\Om$, say $\ell< \frac{1}{64}$. Here, as before, we agree to take $x_0 = 0$, and the balls in the chain can be so chosen that $x_1 \in B_{5r_0/16}$, $x_{i+1} \in B_{\frac{\ell r_0}{2}}(x_i)$ for $i = 1,...,d-1$, and $\overline{x} \in B_{\ell r_0}(x_d)$. We note that $d$ depends on $r_0$, as well as $\Om$. Moreover, since $\Om$ is $C^{1, Dini}$ and hence in particular a Lipschitz domain, one can ensure that the balls $B_{3\ell r_0}(x_i)$ are at a distance  at least $\ell r_0$ from $\partial \Om$.

Since $\ell < \frac{1}{64}$, it is easy to check by triangle inequality that  $B_{\ell r_0/2}(x_1) \subset \Om \cap B_{3r_0/8}$. We thus find from \eqref{g34bis}
\begin{equation}\label{e'31}
||u||_{L^{\infty}(B_{\ell r_0/2}(x_1))} \leq C_1 \ve^{\theta} \exp( C(\sqrt{M} +1)).
\end{equation}
Since the balls $B_{3\ell r_0}(x_i)$  are at a distance comparable to $r_0$ from the  boundary, and since $B_{\ell r_0/2}(x_{i+1}) \subset B_{\ell r_0}(x_i)$ by the triangle inequality, we can now iterate the estimate \eqref{e'31} by using the interior $L^{\infty}$ three-ball theorem as in Lemma 3 in \cite{Zhu1} with $r_1= \ell r_0/2, r_2= \ell r_0, r_3= 3\ell r_0$. Since the balls $B_{3\ell r_0}(x_i)$  are at a distance comparable to $r_0$ from the  boundary, and since $B_{\ell r_0/2}(x_{i+1}) \subset B_{\ell r_0}(x_i)$ by the triangle inequality, we can iterate the estimate \eqref{e'31} by using the interior analogue of \eqref{g35}. After the  $d$-th iteration we find
\begin{equation}\label{e32}
||u||_{L^{\infty}(B_{\ell r_0}(x_d))} \leq C_3 \ve^{\theta^*} \exp(C_4 (\sqrt{M} +1)),
\end{equation}
where the constants $C_3, C_4, \theta^* $ additionally depend on $d$, which in turn depends on $r_0$. Since $\overline{x} \in B_{\ell r_0}(x_d)$, and $|u(\overline{x})| \geq 1$, we obtain from \eqref{e32} 
\begin{equation}\label{e33}
1 \leq ||u||_{L^{\infty}(B_{\ell r_0}(x_d))} \leq C_3 \ve^{\theta^*} \exp(C_4 (\sqrt{M} +1)).
\end{equation}
We thus conclude that  \eqref{h34} holds.

\medskip

\noindent \textbf{Case $2$:} Suppose $d_1 < \frac{3r_0}{16}$, and let $z_0 \in \Gamma \cap \overline{B}_{\frac 32}$ be such that $d_1= |\overline{x}- z_0|$. In this case we take a sequence of balls  centered at $0, y_1, y_2, .....y_d \in \Gamma$ such that $y_1 \in \Om \cap B_{3r_0/16}$, $y_{i+1} \in \Om \cap B_{3r_0/16}(y_i)$ for $i =1,...,d-1$, and $z_0 \in \Om \cap B_{3r_0/16}(y_d)$. Note that $d$ again depends on $r_0$ and $\Om$.  We first observe that \eqref{g34bis} holds as for Case $1$. Moreover, the triangle inequality gives $\Om \cap B_{3r_0/16}(y_1) \subset \Om \cap B_{3r_0/8}$. Combining this with \eqref{g34bis} we obtain
\begin{equation}\label{f33}
||u||_{L^{\infty}(\Om \cap B_{3r_0/16}(y_1))} \leq C \ve^{\theta} \exp( C(\sqrt{M} +1)).
\end{equation}
Now by using the fact that $\Om \cap B_{3r_0/16}(y_{i+1}) \subset \Om \cap B_{3r_0/8}(y_i)$ for each $i$, we can iterate \eqref{f33} by using \eqref{g35} with $z$ replaced by $y_i$, for $i=1,...,d-1$,  and obtain after the $d-$th iteration 
\[
||u||_{L^{\infty}(\Om \cap B_{3r_0/8}(y_d))} \leq C_5 \ve^{\theta^{**}} \exp( C_6 (\sqrt{M} +1)),
\]
where as in Case 1, the constants $C_5, C_6, \theta^{**}$ additionally depend on $d$, which in turn depends on $r_0$. Since $z_0 \in \Om \cap  B_{3r_0/16}(y_d)$ and  $|\overline{x}- z_0| <  \frac{3r_0}{16}$, by the triangle inequality we see that $\overline{x} \in   \Om \cap B_{3r_0/8}(y_d))$. Combining this observation with the fact that   $|u(\overline{x})| \geq 1$, we conclude from the latter inequality that
\[
1 \leq ||u||_{L^{\infty}(\Om \cap B_{3r_0/8}(y_d))} \leq C_5 \ve^{\theta^{**}} \exp( C_6 (\sqrt{M} +1)).
\]
Therefore, in both Cases $1$ and $2$ we see that \eqref{h34} holds.

\end{proof}

The next Lemma \ref{imp2} is the central result of this section. As the reader will see its proof is quite involved. This is unavoidable since we are dealing with variable coefficients, and keeping uniformity matters under control is more delicate than for the standard Laplacian, when $A(x) \equiv I_n$. More specifically, in order to apply Theorem \ref{vimp} above for balls centered at an appropriately chosen $y_0$, we need to use the transformation $T_{y_0}$ in \eqref{T} as an intermediate step to ensure that $A(y_0)=I_n$. The payoff of this is reflected in \eqref{cont}. However, by far \eqref{cont} alone does not suffice to derive our basic estimate \eqref{he} below. We need to crucially use the fact that $T_{y_0}$ is sufficiently close to the identity, in a precise quantitative way, at any given small scale $r$. This is possible thanks to the Lipschitz character of the matrix $A$. We also note that  there will be several intermediate functionals in the proof of Lemma \ref{imp2} below. Their introduction has been necessary to ensure the positivity in the transformed domains of the various weights appearing in the relevant integrands.

In the sequel we will need the following quantity
\[
G(s)= \int_{\Om \cap B_s} u^2 (s^2 - |x|^2)^{\alpha},\ \ \ \ \ \ \ \ \ \ s>0.
\]

\begin{lemma}\label{imp2}
Let $\Om, A(x), u$ and $V$ be  as in Theorem \ref{main}, where we additionally assume  that  $x_0=0$ and $A(0)=I_n$. Then, there exist $R_0<1$ and constants $k, K_1, C, C_1$, depending on $n, \lambda, K, C_0$ and the $C^{1,\operatorname{Dini}}$-character of $ \Om$, but independent of $M$, such that for $0<r<R_0$ one has
\begin{equation}\label{he}
\log\frac{G(r/2)}{G(r/4)} \leq C \sqrt{M}r + e^{C_1 r} \frac{\log \frac{(1+ 4K_1\Lambda(r))(2+ 16 k \Lambda(r))}{(1 -4K_1\Lambda(r))(1-16 k\Lambda(r))}}{\log \frac{(1 -4K_1\Lambda(r))(2- 8k\Lambda(r))}{(1+ 4K_1\Lambda(r))(1+ 8 k \Lambda(r))}} \log\frac{G(r)}{G(r/2)}.
\end{equation}
\end{lemma}

\begin{proof}

We let $a= 4\Lambda(r)r$ and consider the interior point $ y_0= - a \nu(0)$ associated with $x_0 = 0$. Then, by Lemma \ref{L:star} above, we know  that $\Om \cap B_{r- a}(y_0)$ is star-shaped for every $0<r<R_0$, with $R_0$ as in \eqref{1000}. 

At this point we need to use Theorem \ref{vimp} with balls centered at $y_0\in \Om$. The problem, again, is that to apply such result we need to know that \eqref{dev1} holds with $z_0 = y_0$. In order to achieve this condition, we argue as in Lemma
\ref{L:vegen} and use the transformation $T_{y_0}$ defined by \eqref{T} above to map $y_0$ to $0 = T_{y_0}(y_0) \in \Om_{y_0}$. The new matrix $A_{y_0}$, defined in $\Om_{y_0}$ by \eqref{cv}, verifies $A_{y_0}(0) = I_n$. 

From the Lipschitz continuity of $y \to A(y)$ and the fact that $|y_0|=a$, we find for $\xi\in \Rn$
\[
|<(A(y_0) - A(0))\xi,\xi>| \le ||A(y_0) - A(0)|| |\xi|^2 \le K |y_0| |\xi|^2 = K a |\xi|^2.
\]
This observation and the hypothesis $A(0)=I_n$ imply
\[
(1 - K a) |\xi|^2\ \leq\ <A(y_0) \xi,\xi>\ \leq\ (1+Ka) |\xi|^2,
\]
for all $\xi\in \Rn$. If $\lambda_{y_0}\in (0,1]$ is a number such that for all $\xi\in \Rn$
\[
\lambda_{y_0} |\xi|^2 \ \leq\ <A(y_0)\xi,\xi>\ \leq\  \lambda_{y_0}^{-1} |\xi|^2,
\] 
then it is clear that 
\[
\lambda_{y_0} \ge 1 - K a,\ \ \ \ \ \ \ \lambda_{y_0}^{-1} \le 1 + K a.
\]
We thus infer from \eqref{1000} that for $0<r<R_0$
\begin{align}\label{nt}
\lambda_{y_0}^{1/2} \geq 1- K_1a,\ \ \ \ \ \ \ \ \ \lambda_{y_0} ^{-1/2} \leq 1+ K_1 a,
\end{align}
for some $K_1$ that is a universal multiple of  $K$ in Theorem \ref{main}. We also note that, similarly to \eqref{cont}, from the definition \eqref{T} of $T_{y_0}$  we obtain
\begin{equation}\label{cont1}
  B_{s\sqrt{\lambda_{y_0}}}(T_{y_0}(p)) \subset T_{y_0}(B_s(p)) \subset  B_{\frac{s}{\sqrt{\lambda_{y_0}}}}(T_{y_0}(p)).
\end{equation}
 Then, from the arguments in the proof of Lemma \ref{L:gens}, we have that after the transformation $T_{y_0}$ is applied, the set $\Om_{y_0} \cap B_{\sqrt{\lambda_{y_0}}( r- a)}$ satisfies the generalized star-shaped assumption with respect to $A_{y_0}$ and $0$. Let $p' = T_{y_0}(0)\in \partial \Om_{y_0}$, and note that, similarly to \eqref{pprime} above, we have 
 \begin{equation}\label{pp}
 |p'| \leq \lambda_{y_0}^{-1/2} a \le  (1+ K_1 a)a \le (1+K_1)a = K_2 a,
 \end{equation} 
since $a<1$ for $0<r<R_0$ by \eqref{1000}. For later purposes we note that for every $x\in \Rn$ we have
\begin{equation}\label{a10}
T_{y_0}(x)= T_{y_0}(0) + (T_{y_0}(x) - T_{y_0}(0)) = p' + A(y_0)^{-1/2} x.
\end{equation}
Since we are assuming $A(0) = I_n$, we have
\begin{align*}
& ||(A^{-1/2}(y_0) - A(0)^{-1/2})x||^2 = (A(y_0)^{-1}x,x> +|x|^2 - 2 <A(y_0)^{-1/2}x,x>
\\
& \le (\lambda_{y_0}^{-1/2} - 1)^2 |x|^2 \le K_1 a |x|^2.
\end{align*}
This estimate shows that
\[
||A^{-1/2}(y_0) - A(0)^{-1/2}|| \le (\lambda_{y_0}^{-1/2} - 1) \le K_1 a.
\]
As a consequence, we can write
\begin{equation}\label{a11}
A(y_0)^{-1/2} = I_n + B,
\end{equation}  
where $B$ is a matrix such that $||B|| \le K_1 a$.
 We now introduce the following notations.
\begin{align}\label{cont2}
\lambda_1= 1- K_1a,\ \ \ \ \ \ \ \ \ \  \lambda_2  =1+ K_1 a,
\end{align}
and let 
\[
r_1= \frac{r/4 - ka}{\lambda_2},\ \ \ \ r_2=\frac{ r/2 + ka}{\lambda_1},\ \ \ \ r_3= \frac{ r - ka}{\lambda_2},
\]
where $k$ is a universal number that depends only on the constant $K$ in \eqref{hypA1} above. Specifically, with $K_2 = 1 + K_1$ being the universal constant in \eqref{pp}, we choose $k$ as follows
\begin{equation}\label{k2}
k = 8 K_2 \left(\frac{K_2}{K_1} + 3\right) = 8(1+K_1)(4+K_1^{-1}).
\end{equation}
With $k$ being fixed as in \eqref{k2}, we now further restrict $R_0$ in \eqref{1000} by assuming that the following condition hold 
\begin{equation}\label{fr3}
\Lambda(R_0) \leq \text{min}\left\{\frac{1}{24 K_1 + 64k}, \frac{1}{1000}\right\}.
\end{equation}
This assumption will be in force for the rest of the paper.

We now want to verify that $0<r_1<r_2<r_3$. First, we note that in order to guarantee that $r_1>0$ it suffices to have
\[
\Lambda(R_0) < \text{min}\left\{\frac{1}{16k}, \frac{1}{1000}\right\},
\]
which is of course ensured by \eqref{fr3}. Incidentally, since obviously $4k>K_1$, \eqref{fr3} also ensures that $\lambda_1 > 0$. Having said this, we notice that regardless the value of $k$ it is always true that $r_1<r_2$. Instead, since from \eqref{1000} we know that $R_0<1$, in order to ensure that $r_2<r_3$ it is easy to verify that it suffices to have
\[
\Lambda(R_0) < \text{min}\left\{\frac{1}{3 K_1 + 16k}, \frac{1}{1000}\right\},
\]
which again is guaranteed by \eqref{fr3}.
Later on, we will want to ensure that there exist universal numbers $1< c_1 <c_2$, and $1<c_3<c_4$, such that for $0<r<R_0$ one has
\begin{equation}\label{rs}
c_1 \le \frac{r_2}{r_1} \le c_2,\ \ \ \ \ \ \ \ \ \ \ \ \ \ c_3 \le \frac{r_3}{r_2} \le c_4.
\end{equation}
Since 
\[
\frac{r_2}{r_1} = \frac{\frac 12+4k\Lambda(r)}{\frac 14-4k\Lambda(r)}\ \frac{1+K_1 a}{1-K_1 a},\ \ \ \ \ \ \ \ \ \frac{r_3}{r_2} = \frac{1 - 4k\Lambda(r)}{\frac 12 + 4k\Lambda(r)}\ \frac{1-K_1 a}{1+K_1 a},
\]
it is easy to verify that if $\Lambda(R_0) < \min\{1/32 k,1/8K_1\}$, then we have
\[
2\le \frac{\frac 12+4k\Lambda(r)}{\frac 14-4k\Lambda(r)} \le 1028,\ \ \ 1 \le \frac{1+K_1 a}{1-K_1 a} \le 2(1+K_1).
\]
In view of \eqref{fr3} we conclude that \eqref{rs} holds with $c_1 = 2, c_2 = 2056 (1+K_1)$. Furthermore, since \eqref{fr3} guarantees that $\Lambda(R_0) \le \min\{1/64 k,1/24K_1\}$, then we also have $a < 4 \Lambda(R_0) \le \frac{1}{6K_1}$, which gives 
\[
\frac 57 \le \frac{1-K_1a}{1+K_1 a}  \le 1,\ \ \ \ \ \ \frac{15}{9} \le \frac{1 - 4k\Lambda(r)}{\frac 12 + 4k\Lambda(r)} \le 2.
\]
We thus conclude that, by assuming \eqref{fr3}, 
then we guarantee that $r_3/r_2$ satisfy \eqref{rs} with $c_3 = 75/72$ and $c_4 = 2$. In conclusion, for any $0<r<R_0$ both inequalities in \eqref{rs} are in force. In addition, we have $0<r_1<r_2<r_3$ for $0<r<R_0$.

Finally, we want to ensure  that $r_3 < \sqrt{\lambda_{y_0}}(r- a)$ for $0<r<R_0$. This can be seen as follows. By using $\sqrt{\lambda_{y_0}} \geq 1 - K_1 a$ and the fact that $0<r<R_0<1$, see \eqref{1000}, we have 
\begin{align*}
& \sqrt{\lambda_{y_0}}( r- a) - r_3 \geq (1- K_1 a) (r-a) - \frac{ r - ka}{1+ K_1 a} 
\notag
\\
& = \frac{k-1 - K_1^2 a r  + K_1^2 a^2)a}{1+K_1 a} > \frac{ a(k-1 - K_1^2 a)}{1+K_1a} \ge 0,
\end{align*}
provided that $k-1- K_1^2 a \ge 0$. Since, as we have observed above, \eqref{fr3} guarantees that $a < \frac{1}{6K_1}$, we see that this inequality is true.
In conclusion, we have proved that provided that $R_0$ be such that \eqref{fr3} hold, we have 
\[
0<r_1<r_2<r_3<\sqrt{\lambda_{y_0}}(r- a),
\]
and moreover \eqref{rs} is in force.

In what follows, to simplify the notation we will write  $H(r), I(r), N(r)$ to indicate the functions introduced in \eqref{vh}, \eqref{vi} and \eqref{genf} above, but relative to the domain $\Om_{y_0}$, the matrix $A_{y_0}$, the potential $V_{y_0}$ (see \eqref{tran} above), the solution $u_{y_0}$ to \eqref{neweq}, and to balls centered at $0$. Thus, for instance,
\[
H(r)= \int_{\Om_{y_0} \cap B_r} u_{y_0}^2  (r^2 - |y|^2)^{\alpha} \mu_{y_0},
\]
where, by slightly abusing the notation introduced in \eqref{mu} above, we have indicated with $\mu_{y_0}$ the conformal factor
\begin{equation}\label{muy0}
\mu_{y_0}(y) = \frac{<A_{y_0}(y)y,y>}{|y|^2}.
\end{equation}
Similarly, we have
\[
I (r)  = \int_{\Om_{y_0} \cap B_r}   <A_{y_0} Du_{y_0}, Du_{y_0}> (r^2 - |y|^2)^{\alpha+1}  + \int_{\Om_{y_0} \cap B_r} V_{y_0} u_{y_0}^2 (r^2 - |y|^2)^{\alpha+1},
\]
and we let $N(r) = I(r)/H(r)$.

Since $A_{y_0}(0) = I_n$, from \eqref{vnbis} above we obtain 
\begin{equation}\label{vnbis0}
\frac{d}{ds} \log\left(\frac{H(s)}{s^{2 \alpha+ n}}\right) = O(1) + \frac{1}{(\alpha +1)s} N(s),\ \ \ \ 0<s<R_0,
\end{equation} 
for some universal $O(1)$ for which, say, $|O(1)| \le \overline C$.
Furthermore, as we have observed above the set $\Om_{y_0} \cap B_{\sqrt{\lambda_{y_0}}( r- a)}$ satisfies the generalized star-shaped assumption with respect to $A_{y_0}$ and $0$. We are thus in a position to apply Theorem \ref{vimp}, which gives for every $0<s<t< \sqrt{\lambda_{y_0}}( r- a)$
\[
e^{C_1s} (N(s) + C_2 M s^2) \le e^{C_1t} (N(t) + C_2 M t^2).
\]
Observing that, since $\lambda_{y_0} \le 1$, we trivially have $0<r_1<r_2<r_3<\sqrt{\lambda_{y_0}}( r- a)<r<R_0<1$, the above estimate implies in particular for every $0<s<r_2$
\begin{equation}\label{int1}
N(s)  \leq e ^{C_1r} ( N(r_2) +  C_2 M r).
\end{equation}
Similarly, for $r_2<s<r$ we obtain, with $C = C_2 e^{C_1}$,
\begin{equation}\label{int2}
N(s)  \geq e ^{-C_1r} ( N(r_2) - C M r).
\end{equation}
Integrating \eqref{vnbis0} on the interval $[r_1,r_2]$, and using \eqref{int1}, we find
\begin{equation}\label{cr1}
\log\ \frac{H(r_2)}{H(r_1)} \leq (2 \alpha  + n) \log\ \frac{r_2}{r_1} +  e^{C_1r} \frac{1}{\alpha+ 1} N(r_2) \log\ \frac{r_2}{r_1} + C \alpha r   \log\ \frac{r_2}{r_1} + \overline C r,
\end{equation}
where we recall that $\alpha = \sqrt M$, and that without loss of generality we have assumed $M\ge 1$.
Similarly, integrating \eqref{vnbis0} on the interval $[r_2,r_3]$, and using \eqref{int2}, we find
\begin{equation}\label{cr2}
\log\ \frac{H(r_3)}{H(r_2)} \geq (2 \alpha  + n) \log\ \frac{r_3}{r_2}  + e^{-C_1 r} \frac{1}{\alpha+ 1} N(r_2) \log\ \frac{r_3}{r_2}  -  C \alpha r \log\ \frac{r_3}{r_2}  - \overline C r.
\end{equation}
Using \eqref{rs} we obtain from \eqref{cr1} 
\begin{equation}\label{cr3}
\frac{N(r_2)}{\alpha +1 } \geq e^{-C_1 r} \frac{\log\ \frac{H(r_2)}{H(r_1)}}{\log\ \frac{r_2}{r_1}} - (2 \alpha + n)e^{-C_1 r} - C \alpha r - \tilde{C} r,
\end{equation}
where $\tilde C = \overline C/\log c_1$. In the same way, \eqref{rs} and \eqref{cr2} give 
\begin{equation}\label{cr4}
\frac{N(r_2)}{\alpha +1 } \leq  e^{C_1 r} \frac{\log\ \frac{H(r_3)}{H(r_2)}}{\log\ \frac{r_3}{r_2}} - (2 \alpha + n)e^{C_1 r} + C e^{C_1} \alpha r + C^\star r,
\end{equation}
where $C^\star = \overline C/\log c_3$.
Therefore,  from \eqref{cr3} and \eqref{cr4}  by  using the fact that $e^{-C_1 r} \leq e^{C_1 r}$, we find for a different $C_3$ which is still universal that the following holds
\[
e^{-C_1 r} \frac{\log\ \frac{H(r_2)}{H(r_1)}}{\log\ \frac{r_2}{r_1}} \leq e^{C_1 r} \frac{\log\ \frac{H(r_3)}{H(r_2)}}{\log\ \frac{r_3}{r_2}} + C_3 (\sqrt{M} +1)r,
\]
which again implies, for a different universal constant $C_4$, that the following holds, 
\begin{equation}\label{L}
\frac{\log\ \frac{H(r_2)}{H(r_1)}}{\log\ \frac{r_2}{r_1}} \leq e^{2 C_1 r} \frac{\log\ \frac{H(r_3)}{H(r_2)}}{\log\ \frac{r_3}{r_2}} + C_4 \sqrt{M} r.
\end{equation}

If we now define  
\[
r'_1= \frac{ r/4 - ka/2}{\lambda_2},\ \ \ r'_2= \frac{r/2 + ka/2}{\lambda_1},\ \ \ \ r'_3= \frac{ r - ka/2}{\lambda_2},
\]
then we claim that with $k$ as in \eqref{k2}, and $R_0$ such that \eqref{fr3} hold,  for $0<r<R_0$ we have the following implications:
\begin{align}\label{A1}
& y \in  B_{r_2'} (p')\ \Longrightarrow\  \left((r_2')^2  - |y- p'|^2\right)^{\alpha} \le ((r_{2})^2 - |y|^2)^{\alpha}, 
\end{align}
\begin{align}\label{A2}
& y \in   B_{r_1}\ \Longrightarrow\ (r_1^2 - |y|^2)^{\alpha} \leq  \left((r_1')^2 - |y- p'|^2\right)^\alpha, 
\end{align}
and
\begin{equation}\label{A3}
 y \in   B_{r_3} \ \Longrightarrow\ (r_3^2 - |y|^2)^{\alpha} \leq  \left((r_3')^2 - |y- p'|^2\right)^\alpha.
\end{equation}

The validity of \eqref{A1} can be seen as follows. We note that from the definition of $\lambda_1$, the implied inequality in \eqref{A1} is equivalent to
\[
\frac{ \frac{3k^2 a^2}{4} + \frac{rka}{2} - 2<y,p'>((1-K_1a))^2 +  |p'|^2 (1-K_1a)^2}{\lambda_1^2} \geq 0.
\]
It is clear that for the latter inequality to be true it suffices to have
\begin{equation}\label{n1}
\frac{3k^2 a^2}{4} + \frac{rka}{2} \ge 2<y,p'> (1-K_1a)^2.
\end{equation}
Since $|p'| \leq K_2 a$, and $|y|\le |y-p'| + |p'| < \frac{r/2 + ka/2}{\lambda_1} + K_2 a$, we have
\[
2<y,p'> (1-K_1a)^2 \le  K_2 a \left(\frac{r + ka}{\lambda_1} + 2K_2 a\right),
\]
and thus for \eqref{n1} to hold it suffices to have
\[
\frac{2K_2}{\lambda_1} + 8 K_2 \Lambda(r) \left(\frac{k}{\lambda_1} + 2K_2\right) \le  k.
\] 
At this point recall that \eqref{fr3} gives, in particular, $a < \frac{1}{6K_1}$. We thus find $\frac 56 \le \lambda_1 \le 1$ for $0<r<R_0$, and thus
\[
\frac{2K_2}{\lambda_1} + 8 K_2 \Lambda(r) \left(\frac{k}{\lambda_1} + 2K_2\right) \le \frac{12}{5} K_2 + \frac{48}{5} K_2 \frac{k}{64 k} + 16 K_2^2 \frac{1}{24K_1}. 
\]
It is thus clear that in order to ensure that \eqref{k2} does hold it suffices to choose 
\[
k \ge K_2 \left(\frac{K_2}{K_1}+ 3\right).
\] 
Similar elementary considerations show that in order to guarantee the implied inequality in \eqref{A2} it suffices to have
\[
\frac{3 k^2 a^2}{4} + |p'|^2 \lambda_2^2 - 2 <y,p'> \lambda_2^2 \le \frac{kar}{4}.
\]
Since by hypothesis $y \in   B_{r_1}$, in order for the latter inequality to hold it suffices to have
\begin{equation}\label{k4}
3 k^2 \Lambda(r) + 4 K_2^2 \lambda_2^2 \Lambda(r) + \frac{K_2 \lambda_2}{2} \le \frac k4.
\end{equation}
However, \eqref{fr3} trivially implies that for every $0<r<R_0$ one has $3 k^2 \Lambda(r)  \le \frac k8$, and $\lambda_2 \le \frac{25}{24}\le \frac 32$. We conclude that the latter inequality is certainly valid if
\[
\frac k8 + \frac{K_2^2}{K_1} + K_2 \le \frac k4,
\]  
which clearly holds provided that 
\[
k \ge 8 K_2 \left(\frac{K_2}{K_1} + 1\right).
\]
Finally, for the implied inequality in \eqref{A3} to hold true it suffices that
\[
\frac{3k^2 a^2}{4} + |p'|^2 \lambda_2^2 + 2<y,p'>\lambda_2^2 \le k a r, 
\]
which in turn is implied by the inequality
\[
 3 k^2 \Lambda(r) + 4 K_2^2 \lambda_2^2 \Lambda(r) + 2 K_2 \lambda_2 + 8 K_2^2 \lambda_2^2 \Lambda(r) \le k.
 \] 
 Since \eqref{fr3} trivially implies that $3 k^2 \Lambda(r) \le \frac k2$, and that as above $\lambda_2 \le \frac 32$, we conclude that the latter inequality is certainly valid if
\[
k \ge 4 K_2 \left(\frac{K_2}{K_1} + 2\right).
\] 
In conclusion, it is immediate to verify that, having chosen the universal number $k$ as in \eqref{k2} above, the implications \eqref{A1}, \eqref{A2} and \eqref{A3} are all true for all $0<r<R_0$, provided that $R_0$ satisfy \eqref{fr3}. Furthermore, it is also easy to check that the assumptions \eqref{k2} and \eqref{fr3} also guarantee the following inclusions:
\begin{equation}\label{incl}
B_{r_2'}(p') \subset B_{r_2},\ \ \ \ B_{r_1} \subset B_{r_1'(p')},\ \ \ \ B_{r_3} \subset B_{r_3'}(p').
\end{equation}
Therefore, if we now set 
\[
L_{p'}(s)= \int_{\Om_{y_0} \cap B_s(p')} u_{y_0}^2 (s^2 - |y-p'|^2)^{\alpha} \mu_{y_0}, 
\]
where $\mu_{y_0}$ is as in \eqref{muy0} above, then from \eqref{A1}, \eqref{A2}, \eqref{A3} and \eqref{incl} we can easily check that 
\[
L_{p'}(r'_2) \leq H(r_2),\ \ \ \  L_{p'}(r'_1) \geq H(r_1),\ \ \ \   H(r_3) \le L_{p'}(r'_3).
\]
Using these inequalities in \eqref{L} and the definitions of $r_1, r_2, r_3, \lambda_1, \lambda_2$ , we find
\begin{equation}\label{iter1}
\log \frac{L_{p'}(r'_2)}{L_{p'}(r'_1)} \leq e^{C_1 r} \frac{\log \frac{(1+ 4K_1\Lambda(r))( 2+ 16 k \Lambda(r))}{(1 -4K_1\Lambda(r))(1-16 k\Lambda(r))}}{\log \frac{(1 -4K_1\Lambda(r) )(2- 8k\Lambda(r))}{(1+ 4K_1\Lambda(r))(1+ 8 k \Lambda(r))}} \log \frac{L_{p'}(r'_3)}{L_{p'}(r'_2)}+ C_4 \sqrt{M} r.
\end{equation}
In passing from \eqref{L} to \eqref{iter1} we have used the fact that since $r <1 $ we trivially have $1 + 4K_1 \Lambda(r)r \leq 1 + 4K_1 \Lambda(r)$, and $1 - 4K_1 \Lambda(r)r \geq 1 - 4K_1 \Lambda(r)$.

Since $\mu_{y_0}$ defined in \eqref{muy0}
 above is Lipschitz at $0$, see Lemma \ref{L:mu}, there exists a universal $C_5$ such that 
\[
(1- C_5 r) \leq \mu_{y_0}(y) \leq (1+ C_5 r)
\]
for $|y| \leq r$. 
This implies that for $s \leq r$,
\begin{equation}\label{iter2}
(1- C_5 r)\tilde{H}(s) \leq L_{p'}(s)  \leq (1+ C_5 r)\tilde{H}(s),
\end{equation}
where we have let
\[
\tilde{H}(s)= \int_{\Om_{y_0} \cap B_s(p')} u_{y_0}^2 (s^2- |y-p'|^2)^{\alpha}.
\]
Applying \eqref{iter1} and \eqref{iter2} with $s=r_1, r_2, r_3$,  we obtain
\begin{equation}\label{iter3}
\log \frac{(1-C_5 r)\tilde{H}(r'_2)}{(1+ C_5r)\tilde{H}(r'_1)} \leq  e^{C_1 r} \frac{\log \frac{(1+ 4K_1\Lambda(r) )(2+ 16 k \Lambda(r))}{(1 -4K_1\Lambda(r) )(1-16 k\Lambda(r))}}{\log \frac{(1 -4K_1\Lambda(r) )(2- 8k\Lambda(r))}{(1+ 4K_1\Lambda(r) )(1+ 8 k \Lambda(r))}} \log\ \frac{(1+ C_5 r)\tilde{H}(r'_3)}{(1-C_5 r) \tilde{H}(r'_2)}+ C_4 \sqrt{M} r.
\end{equation}
Since $0<r<R_0$ and \eqref{fr3} is in force, \eqref{iter3} implies  for a  universal $C_6$
\begin{equation}\label{iter4}
\log \frac{\tilde{H}(r'_2)}{\tilde{H}(r'_1)} \leq  e^{C_1 r} \frac{\log \frac{(1+ 4K_1\Lambda(r))( 2+ 16 k \Lambda(r))}{(1 -4K_1\Lambda(r))(1-16 k\Lambda(r))}}{\log \frac{(1 -4K_1\Lambda(r) )(2- 8k\Lambda(r))}{(1+ 4K_1\Lambda(r))(1+ 8 k \Lambda(r))}}\log \frac{\tilde{H}(r'_3)}{\tilde{H}(r'_2)}+ C_4 \sqrt{M} r  + C_6 \log \frac{1+C_5r}{1-C_5 r}.
\end{equation}
We now use the fact that there exists a universal constant $C_7$ such that
\[
\log \frac{1+C_5 r}{1-C_5 r} \leq C_7 r.
\]
Using this estimate in \eqref{iter4}, we finally obtain 
\begin{equation}\label{iter10}
\log \frac{\tilde{H}(r'_2)}{\tilde{H}(r'_1)} \leq  e^{C_1 r} \frac{\log \frac{(1+ 4K_1\Lambda(r) )( 2+ 16 k \Lambda(r))}{(1 -4K_1\Lambda(r))(1-16 k\Lambda(r))}}{\log \frac{(1 -4K_1\Lambda(r) )(2- 8k\Lambda(r)}{(1+ 4K_1\Lambda(r))(1+ 8 k \Lambda(r))}} \log \frac{\tilde{H}(r'_3)}{ \tilde{H}(r'_2)}+ C_4 \sqrt{M} r  + C_7 r.
\end{equation}
Note that since we  are assuming  that $M \ge 1$, we can absorb the term  $C_7 r$ in $C_4 \sqrt{M}r$ . 

At this point we define 
\[
r''_1= r/4 - ka/2 = \lambda_2 r_1',\ \ \ \ r''_2= r/2 + ka/2 = \lambda_1 r_2',\ \ \ \ r''_3=  r - ka/2 = \lambda_2 r_3',
\]
and
\[
r'''_1= r/4 - ka/3,\ \ \ \ r'''_2= \frac{r/2 + ka/3}{\lambda_1},\ \ \ \ r'''_3=  r - ka/3.
\]
We note that,  since  $r'_1 \leq r'''_1, r'_2 \geq r'''_2,  r'_3 \leq r'''_3$, if we introduce the quantities 
\begin{align*}
&  H_1(r'_1) = \int_{\Om_{y_0} \cap B_{r'_1}(p')} u_{y_0}^2 ((r'''_1)^2 - |y-p'|^2)^{\alpha},
\\
& H_2(r'''_2)= \int_{\Om_{y_0} \cap B_{r'''_2}(p')} u_{y_0}^2 ( r'_2)^2 - |y-p'|^2)^{\alpha},
\\
&  H_3(r'_3) = \int_{\Om_{y_0} \cap B_{r'_3}(p')} u_{y_0}^2 ((r'''_3)^2 - |y-p'|^2)^{\alpha},
\end{align*}
then from \eqref{iter10} we obtain for a universal constant $C_8$,
\begin{equation}\label{iter11}
\log \frac{H_2(r'''_2)}{H_1(r'_1)} \leq  e^{C_1 r} \frac{\log \frac{(1+ 4K_1\Lambda(r))( 2+ 16 k \Lambda(r))}{(1 -4K_1\Lambda(r))(1-16 k\Lambda(r))}}{\log \frac{(1 -4K_1\Lambda(r) )(2- 8k\Lambda(r))}{(1+ 4K_1\Lambda(r))(1+ 8 k \Lambda(r))}} \log \frac{H_3(r'_3)}{ H_2(r'''_2)}+ C_8 \sqrt{M} r.
\end{equation}

We now return to the original domain $\Om$ by applying the  transformation $T_{y_0}^{-1}$.  We note from \eqref{a11} that  the Jacobian $JT_{y_0}$ of the transformation $T_{y_0}$ satisfies $|JT_{y_0}| = 1 + O(a)$, where $O(a)$ is universal.
Therefore, there exists a universal constant $C_9$ such that
\begin{equation}\label{jc}
1- C_9 r \leq |JT_{y_0}| \leq 1 + C_9 r.
\end{equation}
Changing the variable $y = T_{y_0}(x)$ in the integrals defining the quantities $H_1(r_1'), H_2(r'''_2)$ and $H_3(r'_3)$, we find
\begin{align}\label{cv}
 &  H_1(r'_1) = \int_{\Om \cap T_{y_0}^{-1}( B_{r'_1}(p'))} u^2  ((r'''_1)^2 - |T_{y_0}(x)-p'|^2)^{\alpha}| |JT_{y_0}|,
\\
& H_2(r'''_2)= \int_{\Om \cap T_{y_0}^{-1}( B_{r'''_2}(p'))} u^2 (( r'_2)^2 - |T_{y_0}(x)-p'|^2)^{\alpha}|JT_{y_0}|,
\notag
\\
&  H_3(r'_3) = \int_{\Om \cap T_{y_0}^{-1}( B_{r'_3}(p'))} u^2 ((r'''_3)^2 - |T_{y_0}(x)-p'|^2)^{\alpha}|JT_{y_0}|.
\notag
\end{align}
Now, by \eqref{nt}, \eqref{cont1} and the definition of $\lambda_{i}, r'_{i}, r''_i, r'''_{i}$, we have that
\begin{align}\label{cont3}
 T_{y_0}^{-1}( B_{r'_1}(p')) \subset  B_{r''_1},\ \ \ B_{r/2 + ka/3} \subset T_{y_0}^{-1}( B_{r'''_2}(p')),\ \ \ T_{y_0}^{-1}( B_{r'_3}(p')) \subset B_{r''_3}.
\end{align}
In fact, proving the first inclusion in \eqref{cont3} is equivalent to showing that $B_{r'_1}(p') \subset T_{y_0}(B_{r''_1})$. Recalling that $p' = T_{y_0}(0)$, we have from \eqref{cont1} $B_{r'_1}(p') = B_{r'_1}(T_{y_0}(0))\subset  T_{y_0}(B_{\lambda_{y_0}^{-1/2} r_1'})$. To establish the first inclusion it thus suffices to check that $\lambda_{y_0}^{-1/2} r_1' \le r_1''$. Recalling that $r_1'' = \lambda_2 r_1'$, it thus suffices to have $\lambda_{y_0}^{-1/2}  \le \lambda_2$, but this is precisely the content of the second inequality in \eqref{nt} if we keep \eqref{cont2} in mind. In a similar way, we see that, setting $s = r/2 + ka/3$, then the second inclusion on \eqref{cont3} is equivalent to $T_{y_0}(B_{s})\subset B_{r_2'''}(p')$. Again from \eqref{cont1} we have $T_{y_0}(B_{s}) \subset B_{s/\sqrt{\lambda_{y_0}}}(T_{y_0}(0)) = B_{s/\sqrt{\lambda_{y_0}}}(p')$, and thus for the second inclusion to hold it suffices to have $\lambda_1 \le \sqrt{\lambda_{y_0}}$. But this is the content of the first inequality in \eqref{nt}. Finally, the third inclusion in \eqref{cont3} is proved exactly as the first one.

From the inclusions \eqref{cont3} and from \eqref{cv} we conclude that the following are true:
\begin{align}\label{cv2}
&  H_1(r'_1) \le \int_{\Om \cap B_{r''_1}} u^2  ((r'''_1)^2 - |T_{y_0}(x)-p'|^2)^{\alpha}| |JT_{y_0}|,
\\
& H_2(r'''_2) \ge \int_{\Om \cap B_{r/2 + ka/3}} u^2 (( r'_2)^2 - |T_{y_0}(x)-p'|^2)^{\alpha}|JT_{y_0}|,
\notag
\\
&  H_3(r'_3) \le \int_{\Om \cap B_{r''_3}} u^2 ((r'''_3)^2 - |T_{y_0}(x)-p'|^2)^{\alpha}|JT_{y_0}|,
\notag
\end{align}
One concern about the integrals appearing in \eqref{cv2} is that, having increased the domains of integration, we might lose control of the sign of the weights appearing in the integrals. This possibility is excluded by the following considerations.

By \eqref{a10}, \eqref{a11} we have for $x \in   B_{r''_1}$
\[
|T_{y_0}(x)-p'| \le (1+K_1 a) |x| < (1+K_1 a) r_1'' = (1+K_1 a)(r/4 - ka/2). 
\]
Since $r<1$, we conclude that for $0<r<R_0$
\[
 r'''_1  - |T_{y_0}(x)-p'| \geq r/4 - ka/3  - (1+K_1 a) r/4 + (1+K_1 a)ka/2  \geq ka/6 -K_1a/4\geq 0
\]
provided that $k\ge 3K_1/2$, which is true thanks to \eqref{k2} above. Similarly, for $x\in B_{r/2 + ka/3}$ we have
\[
r'_2 - |T_{y_0}(x)-p'| \geq \frac{r/2 + ka/2}{\lambda_1} - (1+K_1 a) r/2 + (1+K_1 a)ka/3 \ge a (5k/6 - K_1/2) \ge 0,
\]
provided that $k\ge 3K_1/5$, which is of course true thanks to \eqref{k2} above.
Finally, for $x \in  B_{r''_3}$ we have
\[
r'''_3 - |T_{y_0}(x)-p'| \geq ka/2 - ka/3 - K_1 a \geq 0
\]
provided that $k\ge 6 K_1$, which is again true by \eqref{k2}.

From \eqref{iter11}, \eqref{cv2} and \eqref{jc} we conclude 
\[
\log \frac{(1-C_9 r)G_2(r/2 + ka/3)}{(1+ C_9 r)G_1(r''_1)} \leq  e^{C_1 r} \frac{\log \frac{(1+ 4K_1\Lambda(r))(2+ 16 k \Lambda(r))}{(1 -4K_1\Lambda(r))(1-16 k\Lambda(r))}}{\log \frac{(1 -4K_1\Lambda(r))(2- 8k\Lambda(r))}{(1+ 4K_1\Lambda(r))(1+ 8 k \Lambda(r))}} \log \frac{(1+ C_9 r)G_3(r''_3)}{(1-C_9 r) G_2(r/2 + ka/3)}+ C_8 \sqrt{M} r,
\]
where we have set
\begin{align*}
G_1(s) & = \int_{\Om_ \cap B_s} u^2 ((r'''_1)^2 - |T_{y_0}x-p'|^2)^{\alpha},
\\
G_2(s) & = \int_{\Om \cap B_s} u^2 ((r'_2)^2 - |T_{y_0}x - p'|^2)^{\alpha},
\\
G_3(s) & = \int_{\Om_1 \cap B_s} u^2 ((r'''_3)^2 - |T_{y_0}x-p'|^2)^{\alpha}.
\end{align*}

Arguing as above we obtain for a universal $C_{10}$
\begin{equation}\label{iter7}
\log\ \frac{G_2(r/2 + ka/3)}{G_1(r''_1)} \leq  e^{Cr} \frac{\log \frac{(1+ 4K_1\Lambda(r))( 2+ 16 k \Lambda(r))}{(1 -4K_1\Lambda(r))(1-16 k\Lambda(r))}}{\log \frac{(1 -4K_1\Lambda(r) )(2- 8k\Lambda(r))}{(1+ 4K_1\Lambda(r))(1+ 8 k \Lambda(r))}} \log\ \frac{G_3(r''_3)}{G_2(r/2 + ka/3)}+ C_{10}\sqrt{M} r.
\end{equation}

In order to complete the proof of the lemma we need to replace the quantities $G_1(r''_1)$, $G_2(r/2 + ka/3)$ and $G_3(r''_3)$ in \eqref{iter7} respectively with
$G(r/4)$, $G(r/2)$ and $G(r)$. With this objective in mind we observe that
\eqref{a10} and \eqref{a11} give with $||B||\le K_1 a$,
\[
|T_{y_0}x-p'|^2 = |x+Bx|^2 = |x|^2 + |Bx|^2 + 2<Bx,x>.
\]
We now claim if $0<r<R_0$ and the conditions \eqref{fr3} and \eqref{k2} are in place, then the following implications  hold
\begin{equation}\label{cv5}
x \in \Om \cap B_{r''_1}  \ \Longrightarrow\ ((r'''_1)^2 - |T_{y_0}x-p'|^2)^{\alpha} \leq ((r/4 - ka/6)^2 - |x|^2)^{\alpha}, 
\end{equation}
\begin{equation}\label{cv6}
x \in \Om \cap B_{r/2 + ka/3}
 \ \Longrightarrow\  ((r'_2)^2 - |T_{y_0}x - p'|^2)^{\alpha} \geq ((r/2+ ka/3)^2 - |x|^2)^{\alpha},
\end{equation}
and
\begin{equation}\label{cv7} 
x \in \Om \cap B_{r''_3}\ \Longrightarrow\ ((r'''_3)^2 - |T_{y_0}x-p'|^2)^{\alpha} \leq ((r - ka/6)^2 - |x|^2)^{\alpha}. 
\end{equation}
The justification of \eqref{cv5}, \eqref{cv6} and \eqref{cv7} follows similarly to that of \eqref{A1}, \eqref{A2} and \eqref{A3} above. For instance, for the implied inequality in \eqref{cv5} to hold true it suffices that
\[
\frac{ka}{3}\left(\frac r4 - \frac{ka}{6}\right) + |Bx|^2 + 2<Bx,x>\ \ge\ 0.
\]
Since $|x| < r/4 - ka/2$, for the latter inequality to be valid it suffices that
\[
4 k^2 \Lambda(r)  + 24 K_1^2 \Lambda(r)(r - 2 ka)^2 
\le k.
\]
Using \eqref{fr3} we can bound from above the left-hand side of this inequality by $k/16 + 2 K_1$ which is of course $\le k$ if, e.g., $k\ge 3 K_1$. Since this is trivially guaranteed by \eqref{k2} above, we conclude that the implication \eqref{cv5} is true. In a similar way, we see that for the implied inequality in \eqref{cv6} to hold it suffices that
\[
\left(\frac r2 + \frac{ka}{2}\right)^2 - \left(\frac r2 + \frac{ka}{3}\right)^2 - \lambda_1^2 K_1^2 a^2 |x|^2 - 2\lambda_1^2 K_1 a |x|^2 - K_1^2 a^2 \left(\frac r2 + \frac{ka}{3}\right)^2 \ge 0.
\]
Since $|x| < r/2 - ka/3$, $\lambda_1 \le 1$, for the latter inequality to be true it suffices to have
\[
\left(2 K_1 + 2 K_1^2 a\right)\left(\frac 12 + \frac{4k\Lambda(r)}{3}\right)^2 \le \frac{k}{6}.
\]
From \eqref{fr3} we immediately see that the left-hand side is bounded from above by $\left(2 K_1 + \frac{K_1}{3}\right)\left(\frac 12 + \frac{1}{48}\right)^2$. It is now obvious that this quantity is $\le k/6$ under the hypothesis \eqref{k2}. Therefore, \eqref{cv6} is true. Finally, we leave it to the reader to verify that \eqref{cv7} does hold under the hypothesis \eqref{fr3} and \eqref{k2}.

At this point,  by taking into account the definition  of $G$, from \eqref{iter7}, \eqref{cv5}, \eqref{cv6} and \eqref{cv7} we conclude that the following holds 
\[
\log \frac{G(r/2 + ka/3)}{G(r/4 - ka/6)} \leq  e^{C_1 r} \frac{\log \frac{(1+ 4K_1\Lambda(r))(2+ 16 k \Lambda(r))}{(1 -4K_1\Lambda(r))(1-16 k\Lambda(r))}}{\log \frac{(1 -4K_1\Lambda(r))(2- 8k\Lambda(r))}{(1+ 4K_1\Lambda(r))(1+ 8 k \Lambda(r))}} \log \frac{G(r- ka/6)}{G(r/2 + ka/3)}+ C_{10} \sqrt{M} r, 
\]
which in turn implies
\[
\log \frac{G(r/2)}{G(r/4)} \leq  e^{C_1 r} \frac{\log \frac{(1+ 4K_1\Lambda(r))(2+ 16 k \Lambda(r))}{(1 -4K_1\Lambda(r))(1-16 k\Lambda(r))}}{\log \frac{(1 -4K_1\Lambda(r))(2- 8k\Lambda(r))}{(1+ 4K_1\Lambda(r))(1+ 8 k \Lambda(r))}} \log \frac{G(r)}{ G(r/2)}+ C_{10} \sqrt{M} r.
\]
This is the desired conclusion \eqref{he}.

\end{proof}

\section{Proof of Theorem \ref{main}}\label{S:gen}

Having established the two main growth Lemmas \ref{L:vegen} and \ref{imp2}, in this section we are finally in a position to provide the
  
\begin{proof}[Proof of Theorem \ref{main}]
With  $x_0$ as in the statement of Theorem \ref{main}, by means of the transformation \eqref{tran} we can assume that $x_0=0$ and $A(0)=I_n$. This would only change the ellipticity bound from $\lambda$ to $\lambda^2$. Also, the Lipchitz norm of the resulting matrix $A_{x_0}$ and of the transformed potential $V_{x_0}$ would differ from that of $A, V$ by some universal factor as in \eqref{VV}-\eqref{newlip}. Therefore, we can reduce the proof to the setting of Lemma \ref{imp2}. In this respect we stress that, henceforth, $k$ and $R_0$ will be fixed as in \eqref{k2} and \eqref{fr3} above. 

We now fix $r_0 = R_0/4$ and iterate \eqref{he} with $r=r_0, r_0/2, r_0/4,...$. After the $q-$th iteration we obtain
\[
\log  \frac{G(r_0/2^{q+1})}{G(r_0/2^{q+2})}  \leq e^{C_1 \sum_{i=1}^q \frac{r_0}{2^i}}\prod_{i=0}^{q}  \frac{\log \frac{(1+ 4K_1 \Lambda(r_0/2^i))(2+ 16 k \Lambda(r_0/2^i))}{ (1- 4K_1 \Lambda(r_0/2^i))(1-16 k\Lambda(r_0/2^i)}}{\log \frac{(1- 4K_1 \Lambda(r_0/2^i))(2- 8k\Lambda(r_0/2^i))}{(1 +4K_1 \Lambda(r_0/2^i))(1+ 8 k \Lambda(r_0/2^i)}}\left[ C \sqrt{M} \sum_{i=0}^q  \frac{r_0}{2^i} + \log \frac{G(r_0)}{G(r_0/2)}\right].
\]
In the latter inequality we can estimate
\[
 C \sqrt{M} \sum_{i=0}^q  \frac{r_0}{2^i}  \leq 2 C R_0 \sqrt{M} \le 2 C \sqrt{M},
\]
where we have used \eqref{1000}. 
We let $\tilde{C}= \text{min} (\frac{1}{24K_1}, \frac{1}{64k})$, and consider the function  $f:[0,\tilde{C}] \to  \R$ defined by
\[
f(y)= \frac{\log \frac{(1+ 4K_1 y)(2+ 16 k y)}{(1- 4K_1y)(1-16 k y)}}{\log \frac{(1-4K_1y)(2- 8k y)}{(1+4K_1y)(1+ 8 k y)}}.
\]
Then, $f\geq 0$,  $f(0)=1$ and $|f'|\le c_2$ on $[0,\tilde{C}]$, for some universal $c_2$. Therefore,
\begin{equation}\label{e53}
0 \leq f(y) \leq e^{c_2 y}. 
\end{equation} 
Since $\Lambda$ is non-decreasing, \eqref{e20} above implies that
\[
\sum_{i=0}^\infty  \Lambda(r_0/2^i)  < \infty.
\]

At this point we note that, thanks to \eqref{fr3}, the number $r_0$ satisfies $\Lambda(r_0) < \tilde{C}$. Combining this observation with \eqref{e53}, we conclude  that
\[
\prod_{i=0}^{\infty}  \frac{\log \frac{(1+ 4K_1 \Lambda(r_0/2^i))(2+ 16 k \Lambda(r_0/2^i))}{(1- 4K_1 \Lambda(r_0/2^i))(1-16 k\Lambda(r_0/2^i)}}{\log \frac{(1- 4K_1 \Lambda(r_0/2^i))(2- 8k\Lambda(r_0/2^i))}{(1 +4K_1 \Lambda(r_0/2^i))(1+ 8 k \Lambda(r_0/2^i)}}
\le \exp\left(c_2 \sum_{i=0}^\infty \Lambda(r_0/2^i)\right) \overset{def}{=} K_0 < \infty.
\]
Using these bounds we obtain for some different universal $C$, 
\begin{equation}\label{e55}
\log  \frac{G(r_0/2^{q+1})}{G(r_0/2^{q+2})} \leq  C \sqrt{M} + K_0 \log \frac{G(r_0)}{G(r_0/2)}.
\end{equation}
We now want to bound from above the quotient $\frac{G(r_0)}{G(r_0/2)}$ in the right-hand side of \eqref{e55}. This will be ultimately accomplished by means of Lemma \ref{L:vegen}. With $C_0 = ||u||_{L^\infty(\Om)} >0$ as in the statement of Theorem \ref{main}, we have the trivial bound
\begin{equation}\label{num}
G(r_0) \le  \omega_n C_0^2 r_0^{2\alpha+n},
\end{equation}
where $\omega_n$ indicates the measure of the unit ball in $\Rn$. 
With $h(r)$ the function defined in \eqref{geh} above, we recall that we have
$h(r) \le (s^2 - r^2)^{-\alpha} H(s)$ for $0<r<s$. Taking $r = 3r_0/8$ and $s = r_0/2$, we obtain
\begin{equation}\label{Hh}
h(3r_0/8) \leq \left(\frac{64}{7}\right)^\alpha r_0^{-2\alpha} H(r_0/2) \leq \lambda^{-1} \overline C^\alpha r_0^{-2\alpha}G(r_0/2),
\end{equation}
 where $\overline C>0$ is an absolute constant. Combining \eqref{num} with \eqref{Hh}, we find for a universal $C>0$ which also depends on $C_0$
 \[
 \frac{G(r_0)}{G(r_0/2)} \le \frac{C \overline C^{\alpha} r_0^n}{h(3r_0/8)}.
 \]
 We are thus left with estimating $h(3r_0/8)$ from below. With this objective in mind we recall the boundary estimates \eqref{best} and \eqref{lle} that give for a certain universal constant $C>0$, depending also on the Lipschitz character of $\partial \Om$,
 \begin{equation}\label{ie}
||u||^2_{L^\infty(\Om \cap B_{r_0/4})} \le  C (1 + ||V||_{L^\infty(\Om)})^2 r_0^{-n} h(3r_0/8).
\end{equation}
At this point we invoke \eqref{h34} in Lemma \ref{L:vegen} above that reads 
\[
||u||^2_{L^\infty(\Om \cap B_{r_0/4})}  \geq L_1 \exp (-L_2(\sqrt{M} + 1)).
\]
Combining this estimate with \eqref{ie}, and observing that $||V||_{L^\infty(\Om)} \le ||V||_{W^{1,\infty}(\Om)} \le M$, we obtain
\[
\frac{G(r_0)}{G(r_0/2)}  \le C(1+M)^2 e^{L_3(\sqrt{M} + 1)},
\]
where $C, L_3$ are universal constants.
Keeping in mind that we are assuming $M\ge 1$, it should be clear to the reader that there exists another universal constant $C$, depending on $n, \lambda, K, C_0$  such that 
\[
\log \frac{G(r_0)}{G(r_0/2)} \le C \sqrt M.
\]
Using this estimate in \eqref{e55}, we finally obtain for every $q\in \mathbb N$
\begin{equation}\label{e58}
\log  \frac{G(r_0/2^{q+1})}{G(r_0/2^{q+2})} \leq \overline C \sqrt{M},
\end{equation}
for a new constant $\overline C>0$ that, depends on $n, \lambda, K,  C_0$, but not on $q$. By a standard argument one easily recognizes that \eqref{e58} implies for every $0< r < r_0$, 
\begin{equation}\label{e59}
G(r) \leq \exp (C \sqrt{M}) G(r/2).
\end{equation}
Again, iterating \eqref{e59} with $r=r_0, r_0/2, ..., r_0/2^q,...$, we find after the $q$-th iteration 
\begin{equation}\label{e60}
G(r_0/2^q) \geq A \exp (- q B \sqrt{M}), 
\end{equation}
for suitable constants $A, B$ independent of $q\in \mathbb N$.
By a standard argument again, the estimate \eqref{e60} easily gives for $0<r< r_0$, 
\begin{equation}\label{e61}
G(r) \geq K_1 r^{C \sqrt{M}}.
\end{equation}
Since as before we have $G(r) \leq C ||u||_{L^{\infty}( B_r \cap \Om)}^2$, the desired conclusion \eqref{m2} in Theorem \ref{main} follows from \eqref{e61}. 

\end{proof}


\begin{thebibliography}{99}

\bibitem[Al]{Al}
F. Almgren, \emph{Dirichlet's problem for multiple valued functions and the regularity of mass minimizing integral currents.}, Minimal submanifolds and geodesics (Proc. Japan-United States Sem., Tokyo, 1977), pp. 1–6, North-Holland, Amsterdam-New York, 1979. 


\bibitem[AE]{AE}
V. Adolfsson \& L. Escauriaza, \emph{$C^{1, \alpha}$ domains and unique continuation at the boundary.} Comm. Pure Appl. Math. \textbf{50}~ (1997), no. 10, 935-969.

\bibitem[AEK]{AEK}
V. Adolfsson,  L. Escauriaza \& C. Kenig, \emph{Convex domains and unique continuation at the boundary}, Rev. Mat. Iberoamericana \textbf{11}~ (1995), no. 3, 513-525.

\bibitem[Bk]{Bk}
L. Bakri, \emph{Quantitative uniqueness for Schr\"odinger operator}, Indiana Univ. Math. J., \textbf{61}~(2012), no. 4, 1565-1580. 

\bibitem[DF1]{DF}
H. Donnelly \& C. Fefferman, \emph{Nodal sets of eigenfunctions on Riemannian manifolds}, Invent. Math, \textbf{93}~ (1988), 161-183.


\bibitem[DF2]{DF1}
H. Donnelly \& C. Fefferman, \emph{Nodal sets of eigenfunctions: Riemannian manifolds with boundary}, Analysis, et cetera, 251Ð262, Academic Press, Boston, MA, 1990. 

\bibitem[GL1]{GL1}
N. Garofalo \& F. Lin, \emph{Monotonicity properties of variational integrals, $A_p$ weights and unique continuation}, Indiana Univ. Math. J. \textbf{35}~(1986),  245-268.

\bibitem[GL2]{GL2}
\bysame, \emph{Unique continuation for elliptic  operators: a geometric-variational approach}, Comm. Pure Appl. Math. \textbf{40}~ (1987),  347-366.

\bibitem[GT]{GT}
D. Gilbarg \& N. S. Trudinger, \emph{Elliptic partial differential equations of second order}, Reprint of the 1998 edition. Classics in Mathematics. Springer-Verlag, Berlin, 2001. xiv+517 pp.

\bibitem[K]{K}
C. Kenig, \emph{Some recent applications of unique continuation},  Recent Developments in Nonlinear
Partial Differential Equations, Contemporary Mathematics,  \textbf{439}~(2007)

\bibitem[Ku1]{Ku}
I. Kukavica, \emph{Quantitative  uniqueness for second order  elliptic operators}, Duke Math. J., \textbf{91}~(1998), 225-240.

\bibitem[Ku2]{Ku2}
I. Kukavica,  \emph{Quantitative, uniqueness, and vortex degree estimates for solutions of the Ginzburg-Landau equation}, Electron. J. Differential Equations \textbf{2000}, No. 61, 15 pp. (electronic).

\bibitem[KN]{KN}
I. Kukavica \& K. Nystrom, \emph{Unique continuation at the boundary for Dini domains},  Proc. Amer. Math. Soc. \textbf{126}~ (1998), no. 2, 441-446.

\bibitem[Li]{Li}
G. M. Lieberman, \emph{The Dirichlet problem for quasilinear elliptic equations with continuously differentiable boundary data}, Comm. Partial Differential Equations, \textbf{11}~(1986), no. 2, 167-229.

\bibitem[N]{N}
J. Necas, \emph{Direct methods in the theory of elliptic equations}, Translated from the 1967 French original by Gerard Tronel and Alois Kufner. Editorial coordination and preface by S\'arka Necasov\'a and a contribution by Christian G. Simader. Springer Monographs in Mathematics. Springer, Heidelberg, 2012. xvi+372 pp.

\bibitem[PW]{PW}
L. E. Payne \& H. F.  Weinberger, \emph{New bounds for solutions of second order elliptic partial differential equations}, Pacific J. Math. \textbf{8}~(1958), 551�573.

\bibitem[Ru]{Ru}
A. R\"uland, \emph{On quantitative unique continuation properties of fractional Schr\"odinger equations: Doubling, vanishing order and nodal domain estimates}, (accepted Trans. AMS),  arXiv:1407.0817


\bibitem[Zhu]{Zhu1}
J. Zhu, \emph{Quantitative uniqueness for elliptic equations},  arXiv:1312.0576




\end{thebibliography}
\end{document}